\documentclass[11pt,reqno]{amsart} 

\usepackage[utf8]{inputenc}
\usepackage{hyperref}
\usepackage{enumitem}
\usepackage[margin=2cm]{geometry}
\usepackage{amsfonts,amsmath,amsthm,amssymb}
\usepackage{graphicx}
\usepackage{siunitx}
\usepackage{mathrsfs} 

\numberwithin{equation}{section}
\newtheorem{theorem}{Theorem}[section]

\newtheorem{lemma}[theorem]{Lemma}
\newtheorem{proposition}[theorem]{Proposition}
\theoremstyle{definition}

\theoremstyle{remark}
\newtheorem{remark}[theorem]{Remark}

\newcommand{\ep}{\varepsilon}
\newcommand{\be}{{\mathrm{b}}}
\newcommand{\crit}{{\mathrm{cr}}}
\newcommand{\tr}{{\mathrm{tr}}}

\title{Large-amplitude steady solitary water waves with constant vorticity}
\author{Susanna V.~Haziot}
\address{Department of Mathematics, Brown University, Box 1917, Providence, RI 02912, USA}
\email{susanna\_haziot@brown.edu}

\author{Miles H.~Wheeler}
\address{Department of Mathematical Sciences, University of Bath, Bath, BA2 7AY, UK}
\email{mw2319@bath.ac.uk}

\date{\today}

\begin{document}
\begin{abstract}
	This paper considers two-dimensional steady solitary waves with constant vorticity propagating under the influence of gravity over an impermeable flat bed. Unlike in previous works on solitary waves, we allow for both internal stagnation points and overhanging wave profiles. Using analytic global bifurcation theory, we construct continuous curves of large-amplitude solutions. Along these curves, either the wave amplitude approaches the maximum possible value, the dimensionless wave speed becomes unbounded, or a singularity develops in a conformal map describing the fluid domain. We also show that an arbitrary solitary wave of elevation with constant vorticity must be supercritical. The existence proof relies on a novel reformulation of the problem as an elliptic system for two scalar functions in a fixed domain, one describing the conformal map of the fluid region and the other the flow beneath the wave.    
	
\end{abstract}	
\maketitle
\tableofcontents
\section{Introduction}
We consider steady two-dimensional waves propagating along the surface of a fluid of finite depth. A great majority of the work done on such waves has been done in the irrotational setting, that is, when the curl of the velocity field vanishes. The mathematical advantage of this setting is that the problem can be reformulated as a problem for a single harmonic function in a fixed domain. For rotational flows, Constantin and Strauss \cite{cs:exact} provided the first existence result for large-amplitude periodic traveling waves. They used a semi-conformal change of variables devised by Dubreil-Jacotin in 1934 \cite{dubreil}, which leads to a quasilinear elliptic problem with a fully nonlinear boundary condition, again for a single function. For a survey on traveling water waves, see \cite{onepas}.

In this paper we study waves which are solitary, that is waves whose profile is a localized -- though not necessarily small -- disturbance of some asymptotic height at infinity. Solitary waves are considerably more difficult to study than their periodic counterparts, since the fluid surface is unbounded which leads to compactness issues. Small-amplitude rotational waves were first constructed by Ter-Krikorov \cite{ter:orig, ter:rot}. The Dubreil-Jacotin change of variables has subsequently been used to construct both small-amplitude \cite{hur:exact, gw} and large-amplitude \cite{paper} waves.

One of the most challenging problems in the theory of steady two-dimensional water waves is proving the existence of waves with \emph{overhanging} profiles. Mathematically, these have a free surface which is not the graph of a function. Numerical results~\cite{ss:deep, sp:steep, vandenb:constant, vandenb:newfamily, dh:numerical, dh:foldsgaps} suggest that such waves occur in the case of non-zero constant vorticity. These waves have both \emph{critical layers} where the horizontal fluid velocity in the moving frame vanishes, and internal \emph{stagnation points} where both components vanish.
Small-amplitude periodic waves which are not overhanging but have critical layers and stagnation points were first constructed in \cite{wahlen:crit}. In 2016, the break-through paper by Constantin, Strauss and Varvaruca \cite{csv:critical} was the first successful attempt at constructing large-amplitude periodic waves with vorticity which can overturn. It is conjectured, based on the numerics, that some of the waves they construct do indeed overturn, but currently there is no rigorous proof. Very recently, overturning periodic waves with constant vorticity have been constructed with weak gravity and large or infinite depth  by perturbing a family of explicit solutions with zero gravity \cite{hw:touching}. Unlike in the periodic case, small-amplitude solitary waves cannot have stagnation points, but see~\cite{kkl:stag}.

The main contribution of this paper is the first construction of large-amplitude solitary waves with constant vorticity which may have overhanging profiles. We prove the existence of a global curve of solutions along which one of several scenarios must occur: the wave amplitude tends to the maximum value so that stagnation is approached at the crest, the dimensionless wave speed blows up, or a singularity forms on the free surface. Our existence results follow from an application of an analytic global bifurcation theorem due to Dancer \cite{dancer} and Buffoni and Toland \cite{bt:analytic}, adapted in \cite{strat} to the unbounded solitary wave setting. As in \cite{csv:critical} we use a conformal mapping, but instead of reformulating the problem as an elegant pseudo-differential equation on the boundary using the Hilbert transform, we work in an elliptic system setting. A significant advantage of this approach, which appears to be new, is that it provides access to linear and nonlinear estimates \cite{adn2,GT,lieberman}. This enables us to strengthen the result in \cite{csv:critical} by, roughly speaking, replacing blow-up in $C^{2+\beta}$ with blow-up in $C^1$. Moreover, this formulation is also less reliant on the vorticity being constant, and may apply to more general vorticity distributions.

\subsection{Presentation of the problem}\label{subsect: Presentation} 
Let $\Omega$ denote the unbounded fluid region in the $(X,Y)$-plane, bounded below by the flat bottom $Y=0$ and above by the water's free surface $\mathcal{S}$. We assume that as $|X|\to\infty$, the fluid domain $\Omega$ approaches a horizontal strip of width $d$. We call $d$ the \emph{asymptotic depth} of the fluid. 

The fact that $\mathcal{S}$ is an unknown is one of the main difficulties of the water wave problem, and it is convenient to introduce a change of variables which fixes the domain. In this paper we assume that $\Omega$ is the image under a conformal mapping $X+iY=\xi(x,y)+i\eta(x,y)$ of the rectangular strip \begin{equation}\label{strip R}
\mathcal{R}:=\{(x,y)\in\mathbb{R}^2:0<y<d\}.
\end{equation} 
We denote the upper and lower boundaries of $\mathcal{R}$ by
\begin{equation*}
	\Gamma:=\{(x,y)\in\mathbb{R}^2:y=d\}\quad\text{and}\quad B:=\{(x,y)\in\mathbb{R}^2:y=0\},
\end{equation*}
respectively. We require that $B$ be mapped onto $Y=0$ and $\Gamma$ onto the free surface $\mathcal{S}$; see Figure~\ref{fig:mapping}. Further details on this map are provided in Section~\ref{sect: Formulation}.  

\begin{figure}
	\centering
	\includegraphics[scale=1.1]{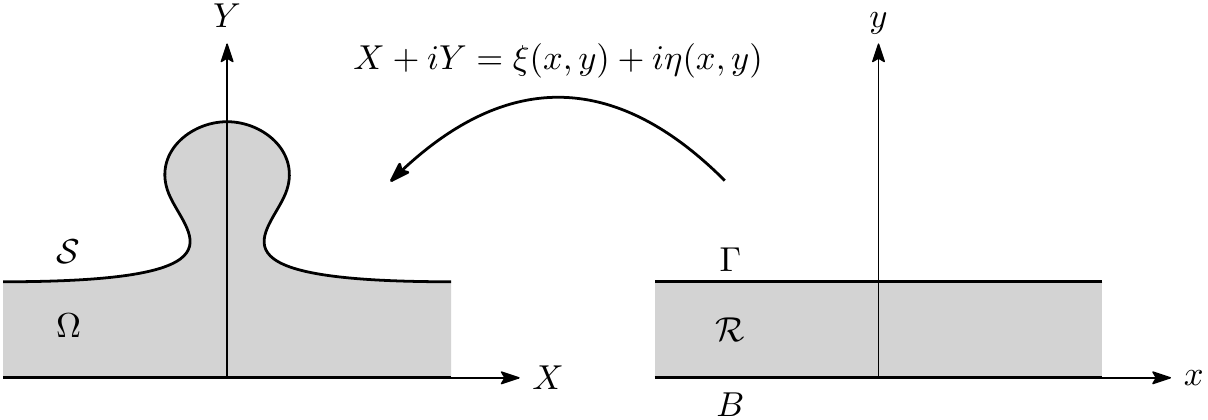}
	\caption{The conformal parametrization of the fluid domain $\Omega$.
		\label{fig:mapping}
	}
\end{figure}

Working in a frame moving with the wave, the fluid velocity $(U,V)$ and pressure $P$ satisfy the two-dimensional steady incompressible Euler equations 
\begin{subequations}\label{euler}
  \begin{alignat}{2}
    \label{incompressibility}
    U_X+V_Y&=0&\qquad&\text{in }\Omega,\\
    \label{horizontal momentum}
    UU_X+VU_Y&=-P_X&\qquad&\text{in }\Omega,\\
    \label{vertical momentum}
    UV_X+VV_Y&=-P_Y-g&\qquad&\text{in }\Omega,
    \intertext{coupled with kinematic boundary conditions}
    \label{kinematic top}
    U\eta_x-V\xi_x&=0&\qquad&\text{on }\mathcal{S},\\
    \label{kinematic bottom}
    V&=0&\qquad&\text{on }Y=0,
    \intertext{on both components of the boundary and the dynamic boundary}
    \label{pressure}
    P&=P_{\text{atm}}&\qquad&\text{on }\mathcal{S}
  \end{alignat}
  on the surface. Here $P_{\text{atm}}$ denotes the constant atmospheric pressure and $g>0$ the constant acceleration due to gravity. We further assume that the vorticity
  \begin{equation}\label{vorticity}
    \omega=\operatorname{curl}(U,V)=U_Y-V_X
  \end{equation}
  is constant throughout the fluid. We note that different authors use different sign conventions in \eqref{vorticity}; here we follow the convention in~\cite{csv:critical}.

\emph{Solitary waves} are solutions to \eqref{horizontal momentum}--\eqref{vorticity} satisfying the asymptotic conditions  
\begin{equation}\label{asymptotic surface}
  \eta(x,d)\to d, 
  \quad 
  \xi(x,d)\to\pm\infty
  \qquad
  \text{as }x\to\pm\infty
\end{equation} 
  on the free surface, and 
  \begin{equation*}\label{asymptotic streamfunction}
    V(X,Y)\to0,
    \quad 
    U(X,Y)\to F\sqrt{gd}\bigg(\gamma\frac{(Y-d)}{d}+1\bigg)
    \qquad\text{as}\,\,X\to\pm\infty
\end{equation*}
on the velocity field, uniformly in $Y$. Here $F$ is the \emph{Froude number}, a dimensionless wave speed, and $\gamma$ is a dimensionless measure of the vorticity $\omega$. We will view $g,d$ and $\gamma$ as fixed and $F$ as the parameter.  
\end{subequations}

We call a solitary wave a \emph{wave of elevation} if $\eta(x,d)>0$ for all $x$. The wave is said to be \emph{symmetric} if $\eta$ is even in the $x$ variable and \emph{monotone} if $\eta_x(x,d)<0$ for $x>0$.
\subsection{Statement of the main results}\label{sec:statement}
The main theorem of the paper is the following existence result. It is stated somewhat informally; we refer the reader to Theorem~\ref{thm: main result full} in Section~\ref{sect: Existence Results} for a precise version. 
\begin{theorem}\label{thm: main result}
  Fix the gravitational constant $g>0$, asymptotic depth $d>0$ and $\gamma<0$. Then there exists a global continuous curve $\mathscr{C}$ of solutions to \eqref{euler} parameterized by $s \in (0,\infty)$. Moreover, the following property holds along $\mathscr{C}$ as $s\to\infty$:
	\begin{equation}\label{main result}
		\min\bigg\{
    \inf_\Gamma\bigg(1- \frac{2}{F^2(s)}\frac{\eta(s)-d}{d}\bigg)
    ,\,
    \inf_\Gamma\big(\eta_x^2(s)+\eta_y^2(s)\big)
    ,\,
    \frac{1}{F(s)} \bigg\}\longrightarrow0.
	\end{equation}
	These solutions are all symmetric and monotone waves of elevation in the sense that $\eta$ is even in $x$ with $\eta_x(x,d)<0$ for $x>0$. 
\end{theorem}

We would like to draw the attention of the reader to the fact that \eqref{main result} only involves pointwise bounds on the geometry of the free surface and the Froude number $F$. In contrast, the analogous result in \cite{csv:critical} for periodic waves involves $C^{2+\beta}$ norms of the conformal map and the stream function. 

Let us now briefly explain the three terms in \eqref{main result}. By symmetry and monotonicity, the infimum in the first term is achieved at the crest $x=0$, where 
\begin{align*}
  \label{eqn:crest}
  1 - \frac 2{F^2} \frac{\eta-d}d = \frac{U^2}{F^2 gd} \ge 0
\end{align*}
is a dimensionless measure of the magnitude of the (purely horizontal) fluid velocity. Provided the Froude number $F$ remains bounded, this first term therefore vanishes precisely when the fluid is approaching stagnation ($U=V=0$) at the crest. This is known to occur in the irrotational case~\cite{at:finite}, where the curve $\mathscr C$ indeed limits to a singular ``wave of greatest height'' with a stagnation point at its crest~\cite{aft}.
The second term in \eqref{main result} is a simple measure of the non-degeneracy of the conformal mapping $\xi+i\eta$. If it were to vanish, we would expect the surface to develop a singularity. The last term of course vanishes only if the Froude number $F$ becomes arbitrarily large. 

Numerical work \cite{vandenb:constant,vandenb:newfamily} suggests that any one of the three terms in \eqref{main result} can vanish. Indeed, when a given dimensionless vorticity parameter is above a certain threshold, these papers find waves of greatest height with negative vorticity. For vorticities below this threshold in the limit $F\to\infty$, Vanden-Broeck shows that the free surface can degenerate into a configuration involving a disk of fluid above a thin strip. Finally, he also generates a family of waves for which $F\to\infty$ but the profile remains regular. 
	
While Theorem~\ref{thm: main result} is stated only for $\gamma < 0$, for the remainder of the paper we make no assumptions on the sign of $\gamma$. However, by adapting the periodic wave argument in \cite{csv:downstream} to our setting, one can show that waves we construct with $\gamma \ge 0$ cannot overturn or have internal stagnation points, and have therefore already been constructed in~\cite{paper} using different methods. Nevertheless, the precise results are not obviously equivalent, and so for completeness we record the analogue of Theorem~\ref{thm: main result} in Proposition~\ref{prop: positive gamma result}.

Another significant result we obtain is the following lower bound on the Froude number. To our knowledge, this is the first bound of this type which applies to waves which are overturning.
\begin{theorem}\label{thm: dim froude}
	Solitary waves of elevation with constant vorticity are supercritical. That is, for all solutions to \eqref{euler} with $\eta\geq d$ and $\eta\not\equiv d$ on the surface $\Gamma$, we have 
	\begin{equation*}
		\frac{1}{F^2}<1-\gamma.
	\end{equation*}
\end{theorem}
\begin{remark}
	In particular, such waves satisfy $\gamma<1$. The proof of Theorem~\ref{thm: dim froude} is based on a useful integral identity derived in Proposition~\ref{prop: integral identity}. The bound is also sharp for small-amplitude waves; for instance see Section~\ref{subsect: small-amplitude}.
\end{remark}

\subsection{Historical considerations}\label{subsect: History}
In 1834, while riding on horseback by a canal, Russell famously saw a solitary wave for the first time \cite{russell}. At the time, this observation was met with little excitement by his peers such as Stokes and Airy as the linear theory they were working with did not allow for such waves \cite{darrigol:horse}. On the contrary, it was even suggested that Russell's wave was simply a periodic wave with a very long wave-length. It was not until the weakly nonlinear long wave asymptotic expansions of Boussinesq in 1877 and Korteweg--de Vries in 1895 that models emerged admitting solitary wave solutions; see \cite{craik}. 

A vast majority of the rigorous work on solitary waves has been done for the irrotational case, and in particular small-amplitude waves. Lavrentiev \cite{lavrentiev} constructed solitary waves as long-wave limits of weakly nonlinear periodic waves. Friedrichs and Hyers \cite{fh} used an inverse function theorem approach which involved rescaling the problem and working in exponentially weighted spaces. Beale \cite{beale} instead used the Nash--Moser inverse function theorem. Mielke \cite{mielke} applied a dynamical systems approach, known as spatial dynamics, which involves looking for a center manifold for an evolution equation in the horizontal spatial variable $X$. 

The first rigorous large-amplitude result for irrotational waves was provided by Amick and Toland \cite{at:periodic, at:finite}, in which they proved the existence of a connected set curve of solutions limiting to stagnation. The authors use a conformal mapping approach and construct a connected set of solitary waves by taking the limit of a sequence of approximate problems with better compactness properties. The limiting wave is singular, which has a stagnation point and \ang{120} angle at its crest \cite{aft}. All the above waves are symmetric monotone waves of elevation with graphical streamlines, precluding the existence of critical layers and overhanging profiles; for more on the qualitative properties of irrotational solitary waves see \cite{amick:bounds, cs:sym, klw:subcritical}. 

For rotational waves, the first small-amplitude result is due to Ter-Krikorov \cite{ter:rot}, who used the method of Friedrichs and Hyers \cite{fh}. This was followed by Hur \cite{hur:exact} who used the Nash--Moser implicit function theorem. Later, Groves and Wahl\'en \cite{gw} constructed small-amplitude solitary waves with an arbitrary vorticity distribution using a spatial dynamics approach. In \cite{hur:symmetry}, Hur then provided the first symmetry results for rotational solitary waves.

Large-amplitude rotational solitary waves were first rigorously constructed in \cite{paper} in which the continuum limits to one of three possible scenarios: vanishing of the horizontal velocity, blow-up of the Froude number, or the Froude number attaining its critical value. The last two alternatives were then eliminated in \cite{froude,strat}. As in \cite{cs:exact}, the waves were constructed using a non-conformal coordinate transformation due to Dubreil-Jacotin \cite{dubreil} which maps the unknown fluid domain into a known rectangular strip by flattening all the streamlines. The resulting problem in the strip consists of a quasi-linear elliptic problem with nonlinear boundary conditions. However, this approach is only applicable to unidirectional flows where all the streamlines are graphs. As a result, the waves constructed in this way cannot have critical layers, internal stagnation points or overhanging profiles.

The first small-amplitude result for periodic rotational waves with internal stagnation points is due to Wahl\'en \cite{wahlen:crit}: using a naive ``surface flattening'' change of variables, he constructed waves containing closed streamlines, referred to as ``Kelvin's cat's-eyes''. Using the same approach, Varholm \cite{varholm:global} was able to successfully construct large-amplitude rotational periodic waves with critical layers. Unlike in \cite{cs:exact}, the possibility that the global curve reconnects is not ruled out. This is related to the fact that the necessary monotonicity properties could not be shown due to a lack of maximum principles. In the interesting paper \cite{kkl:stag}, Kozlov, Kuznetsov and Lokharu construct solitary waves with a critical layer and stagnation point beneath the wave crest in a carefully chosen scaling limit where the vorticity becomes large. Although all the above results provide flows with critical layers and internal stagnation points, the nature of the flattening transformation prevents these waves from overturning.

In \cite{cv:constant}, Constantin and Varvaruca developed an approach which generalizes the conformal map used for irrotational waves to the constant vorticity setting, and constructed small-amplitude rotational periodic waves with critical layers. The main feature of this transformation is that it does not place any restrictions on the geometry of the fluid domain and allows for stagnation points in the flow. As a result, Constantin, Strauss, and Varvaruca  \cite{csv:critical} were then able extend the family of local curves of solutions in \cite{cv:constant} to produce the first large-amplitude periodic waves with constant vorticity which may have overhanging profiles. In order to achieve this, they used the periodic Hilbert transform to reformulate the problem as a quasi-linear pseudo-differential equation in the spirit of Babenko~\cite{babenko}. The main result of this paper consists of the existence of families of continuous curves of solutions which limit to one of four different alternatives: the formation of a wave of greatest height with a stagnation point at the crest, blow up of either the elevation of the wave profile in the Hölder space $C^{2+\beta}$ or of the parameters, and self-intersection of the free surface above the trough line.        

In this paper, as in \cite{csv:critical}, we fix the domain using a conformal map. However, we carry out the analysis in the elliptic system setting rather than using a nonlocal Babenko-like formulation. Working with a local formulation enables us to adapt existing techniques~\cite{volpert:book,paper,pressure,strat,cww:globalfronts} for dealing with issues related to compactness, which as mentioned above is one of the central difficulties in the solitary wave problem. Moreover, the elliptic setting gives us access to powerful linear and nonlinear estimates~\cite{adn2,GT,lieberman}, which enables us to obtain stronger conclusions than in \cite{csv:critical}. Finally, this approach can be easily generalized to allow for an arbitrary vorticity distribution, and so may have further applications, both for periodic waves and solitary waves.

\subsection{Outline of the paper} 

One of the main difficulties of the water wave problem lies in the fact that it is a free boundary problem with nonlinear boundary conditions. As mentioned previously, we resolve this by reformulating the problem in an unbounded strip by means of a conformal map $\xi+i\eta$. We obtain an elliptic system of two harmonic functions $\eta$ and $\zeta$ with fully nonlinear boundary conditions. Here, the function $\zeta$ is closely related to the stream function.  

While it fixes the domain, this transformation gives rise to further difficulties. Most notably, the dynamic boundary condition is no longer oblique in the classical sense of \cite{GT}. Indeed, in \cite{csv:critical} the authors remark that for the analogous nonlocal formulation of this boundary condition, compactness properties seem to be unavailable. Nevertheless, by treating the problem as an elliptic system we are able to obtain Schauder estimates by turning to \cite{adn2}. More precisely, we prove that under the physically meaningful assumption of having no stagnation points on the surface, the linearized boundary operator satisfies the necessary complementing condition. 

However, these linear Schauder estimates are not sufficient for our purposes, and we need nonlinear estimates as well. While we were unable to find general results for elliptic systems with fully nonlinear boundary conditions, we overcome this by applying classical regularity theorems \cite{lieberman} to two carefully chosen coupled scalar problems. More specifically, we show that the $C^{3+\beta}$ norm of the wave elevation is controlled by its $C^1$ norm, thus ensuring that the main theorem of this paper is a significant improvement over the one obtained in \cite{csv:critical}. 

Even with linear and nonlinear estimates, determining Fredholm indices of the linearized version of the operator equation corresponding to our problem is more subtle than in the scalar case, especially for the large-amplitude solutions. Indeed, the complementing condition for elliptic systems is more topologically complicated than obliqueness. As a result, we cannot immediately exclude the possibility of the linearized operator having a non-zero index. We resolve this issue by showing that, in order to apply the global bifurcation theorem for solitary waves as presented in \cite{strat}, the linearized operator only needs to be semi-Fredholm, provided it is invertible and hence Fredholm of index $0$ along the local curve of solutions.

Another significant contribution in this paper concerns sharp lower bounds on the Froude number. Previous results for rotational waves \cite{froude,klw:subcritical} rely heavily on the Dubreil-Jacotin transformation and thus on the absence of internal stagnation points. These break down in our setting. Instead, we give a different argument centered on an auxiliary function related to the \textit{flow force flux function} first introduced in \cite{klw:subcritical}. The bounds we obtain are for waves of elevation only; it would be interesting to see if this assumption can be removed. 

The plan of the paper is the following. In Section~\ref{sect: Formulation} we reformulate the problem using a conformal mapping and non-dimensionalize so that only the parameters $F$ and $\gamma$ remain. Section~\ref{sect: Flow Force} collects arguments involving an invariant known as the ``flow force''. We use it in Section~\ref{subsect: Nonexistence of Bores} to prove, in the spirit of \cite{strat}, that for monotone solutions, compactness properties can only fail in the presence of bores, whose existence we subsequently rule out. In Section~\ref{subsect: Froude number}, we provide lower bounds on the Froude number. In Section~\ref{sect: Nodal Analysis}, we prove that the solutions we construct satisfy a set of monotonicity properties necessary for the compactness results in Section~\ref{subsect: Nonexistence of Bores} to hold. We provide a functional analytic formulation of the problem in Section~\ref{sect: Functional Analytic} along with properties of the linearized operator, and in Section~\ref{sect: Uniform Regularity} we prove our uniform regularity result. In Section~\ref{sect: Existence Results} we first construct small amplitude solitary wave solutions to our problem using the approach in \cite{chen2019center} which allows us, with minimal tedium, to reduce the problem to an ODE. Finally, we construct large-amplitude solutions, thus proving the main theorem of the paper.

Appendix~\ref{app: schauder estimates} specializes the Schauder estimates in \cite{adn2} to a class of elliptic systems which is sufficient for our needs. Appendix~\ref{app: cited results} records our adapted version of the global bifurcation theorem with a brief outline of the proof.  

\subsection{Notation}\label{subsect: notation}
Before we proceed, we briefly define some function spaces we will use throughout the paper. Let $\Omega$ be a connected, open, possibly unbounded subset of $\mathbb{R}^n$. We say that $\varphi\in C_\mathrm{c}^{\infty}(\overline{\Omega})$ if $\varphi\in C^{\infty}$ and the support of $\varphi$ is a compact subset of $\overline{\Omega}$. For $\beta\in(0,1)$ and $k\in\mathbb{N}$ we denote by $C^{k+\beta}(\overline{\Omega})$ the space of functions whose partial derivatives up to order $k$ are Hölder continuous in $\overline{\Omega}$ with exponent $\beta$. We say that $u_n\to u$ in $C_\text{loc}^{k+\beta}(\overline{\Omega})$ if $\|\varphi(u_n-u)\|_{C^{k+\beta}(\Omega)}\to0$ for all $\varphi\in C_\mathrm{c}^{\infty}(\overline{\Omega})$. Moreover, we let $C_\be^{k+\beta}(\overline{\Omega})$ be the Banach space of functions $u\in C^{k+\beta}(\overline{\Omega})$ such that $\|u\|_{C^{k+\beta}(\Omega)}<\infty$. 

When $\Omega$ is unbounded, we denote by $C_0^k(\overline{\Omega})\subseteq C^k_\be(\overline{\Omega})$ the closed subspace of functions whose partial derivatives up to order $k$ vanish uniformly at infinity, that is
\begin{equation}\label{C0}
  C_0^k(\overline{\Omega}):=\bigg\{u\in C_\be^k(\overline \Omega): \lim_{r\to\infty}\sup_{|x|=r}|D^ju(x)|=0\text{ for }0\leq j\leq k \bigg\}.
\end{equation}
We also define weighted Hölder spaces allowing for exponential growth in the $x_1$ direction. Let $C_\mu^{k+\beta}(\overline{\Omega})$ be the space of functions $u\in C^{k+\beta}(\overline \Omega)$ with $\|u\|_{C_\mu^{k+\beta}(\Omega)}<\infty$, where
\begin{equation*}
	\|u\|_{C_\mu^{k+\beta}(\Omega)}:=\sum_{|\alpha|\leq k}\|\operatorname{sech}(\mu x_1)\partial^\alpha u\|_{C^0(\Omega)}+\sum_{|\alpha|=k}\|\operatorname{sech}(\mu x_1)|\partial^\alpha u|_\beta\|_{C^0(\Omega)}.
\end{equation*}
Here $\mu>0$ and $|u|_\beta$ is a local Hölder seminorm
\begin{equation*}
	|u|_\beta(x):=\sup_{|y|<1}\frac{|u(x+y)-u(x)|}{|y|^\beta}.
\end{equation*} 

For the remainder of the paper, we fix the Hölder exponent $\beta\in(0,1)$, once and for all.

\section{Formulation} \label{sect: Formulation}
\subsection{Stream function formulation}
Due to the incompressibility condition \eqref{incompressibility}, we can introduce a \emph{stream function} $\Psi$ defined by
\begin{equation}\label{stream function Psi}
	\Psi_X=-V,\qquad\Psi_Y=U.
\end{equation} 
The kinematic boundary conditions \eqref{kinematic top}--\eqref{kinematic bottom} imply that $\Psi$ is constant on the free surface and on the bed.  We can hence normalize $\Psi$ such that
\begin{equation*}
	\Psi=m\quad\text{on }\mathcal{S}\qquad\text{and}\qquad\Psi=0\quad\text{on }Y=0,
\end{equation*} 
for some constant $m$ which we call the \emph{mass flux}. Taking the curl of \eqref{horizontal momentum} and \eqref{vertical momentum} we see that 
\begin{equation*}
  \Delta\Psi = \omega.
\end{equation*}
Finally, Bernoulli's law states that 
\begin{equation}\label{bernoulli}
  P+\tfrac{1}{2}|\nabla\Psi|^2+gY-\omega\Psi=\text{constant}\qquad\text{in }\Omega,
\end{equation} 
as can be verified by differentiation. Combining all the above considerations, we get
\begin{subequations}\label{originalproblem}
\begin{alignat}{2}
\label{psi harmonic}
\Delta\Psi&=\omega&\qquad&\text{in}\quad\Omega,\\
\label{psi top}
\Psi&=m&\qquad&\text{on}\quad\mathcal{S},\\
\label{psi bottom}
\Psi&=0&\qquad&\text{on}\quad Y=0,\\
\label{psi dynamic}
\vert\nabla\Psi\vert^2+2g(Y-d)&=Q&\qquad&\text{on}\quad\mathcal{S}.
\end{alignat}
\end{subequations}
Here \eqref{psi dynamic} is the result of evaluating \eqref{bernoulli} on the surface $\mathcal{S}$, and we call $Q$ the \emph{Bernoulli constant}. 

A solitary wave solution to \eqref{originalproblem} must also solve the asymptotic conditions \eqref{asymptotic surface} on the free surface along with 
\begin{equation}\label{asymptotic streamfunction 2}
\Psi_X\to0,\qquad \Psi_Y\to F\sqrt{gd}\bigg(\gamma\frac{Y-d}{d}+1\bigg)\qquad\text{as}\,\,X\to\pm\infty,
\end{equation}
uniformly in $Y$. Moreover, using \eqref{asymptotic streamfunction 2} it is easy to see that the Bernoulli constant $Q$ and the mass flux $m$ are given in terms of the Froude number by 
\begin{equation}\label{Q and m}
	Q=F^2gd\qquad\text{and}\qquad m=Fg^{1/2}d^{3/2}\big(\tfrac{1}{2}\gamma-1\big),
\end{equation}
while and the dimensionless measure $\gamma$ of the vorticity $\omega$ is given by
\begin{equation*}
	\gamma=\omega\frac{g^{1/2}}{F d^{1/2}}.
\end{equation*}

\subsection{Conformal mapping}\label{subsect: conformal map}
We now wish to view $\Omega$ as the conformal image of an infinite strip. Let $\mathcal{R}$ be defined as in \eqref{strip R} and $X+iY=\xi(x,y)+i\eta(x,y)$ be a conformal map from $\mathcal{R}$ to $\Omega$, which we normalize by requiring
\begin{equation*}
\xi_x+i\eta_x\to1\quad\text{as }x\to\infty.
\end{equation*}
Since both domains are simply connected, the existence of this transformation is straightforward, unlike the analogue for periodic surfaces as outlined in \cite{cv:constant}. Moreover, the mapping is constructed such that surface of the strip $\Gamma$ is mapped to the free surface which can hence be parameterized as
\begin{equation}\label{surface parametrization}
\mathcal{S}=\{(\xi(s,d),\eta(s,d)): s\in\mathbb{R} \}.
\end{equation}
We will consider solutions to problem \eqref{euler} for which $\Omega$ is of class $C^{3+\beta}$. By Theorem 3.6 in \cite{pommerenke:book}, known as the Kellogg--Warschawski theorem, we conclude that $\xi$ and $\eta$ are of class $C^{3+\beta}(\overline{\mathcal{R}})$. We will see at the end of Section~\ref{sect: Existence Results} that solutions to the problem in conformal variables \eqref{fullproblem} will also yield solutions to \eqref{euler} as claimed in Theorem~\ref{thm: main result}. 

We now rewrite \eqref{originalproblem} in conformal variables. For convenience, we define the stream function in the strip by
\begin{equation}\label{conformal streamfunction}
	\psi(x,y):=\Psi(\xi(x,y),\eta(x,y)).
\end{equation}
So that we can work with harmonic functions, we define the function $\zeta:\mathcal{R}\to\mathbb{R}$ by
\begin{equation}\label{zeta dim}
\zeta(x,y):=\psi(x,y)-\tfrac{1}{2}\omega\eta^2(x,y).
\end{equation}
Expressing \eqref{originalproblem} in these new variables yields
\begin{subequations}\label{fullproblem dim}
	\begin{alignat}{2}
	\label{harmonic zeta dim}
	\Delta\zeta&=0&\qquad&\text{in }\mathcal{R},\\
	\label{kinematic dim}
	\zeta&=m-\tfrac{1}{2}\omega \eta^2&\qquad&\text{on }\Gamma,\\
	\label{dynamic dim}
	(\zeta_y+\omega \eta\eta_y)^2&=(Q-2g (\eta-d))|\nabla\eta|^2&\qquad&\text{on }\Gamma,\\
	\label{bottom zeta dim}
	\zeta&=0&\qquad&\text{on }B,
	\end{alignat}
	which we couple with the conditions
	\begin{alignat}{2}
	\label{harmonic eta dim}
	\Delta\eta&=0&\qquad&\text{in }\mathcal{R},\\  
	\label{bottom eta dim}
	\eta&=0&\qquad&\text{on }B,
	\end{alignat}
	on the imaginary part of the conformal map. Notice that solving for $\eta$ will also provide us, up to a constant, with the harmonic conjugate $\xi$. Finally, the asymptotic conditions \eqref{asymptotic surface} and \eqref{asymptotic streamfunction 2} become 
	\begin{align}
	\label{asymptotic eta dim}
	\lim_{x\to\pm\infty}\zeta(x,y)=\big(m-\tfrac{1}{2}\omega d^2\big)\frac{y}{d}\quad\text{and}\quad\lim_{x\to\pm\infty}\eta(x,y)=y,
	\end{align}
	uniformly in $y$. 
\end{subequations}

\subsection{Final reformulation}\label{subsect: nondim}
First, we reduce \eqref{fullproblem dim} to a one-parameter problem. In order to achieve this we rescale, in both the $(x,y)$ and the $(X,Y)$ variables, the lengths by $d$ and the velocities by $F\sqrt{gd}$. Using \eqref{Q and m}, the process follows in a straightforward way. To simplify notation, from this point forth, we set 
\begin{equation*}\label{alpha}
	\alpha=\frac{1}{F^2},
\end{equation*}
which we will refer to as the \textit{wave speed parameter}. The critical Froude number appearing in Theorem~\ref{thm: dim froude} corresponds to 
\begin{equation}\label{alpha crit}
	\alpha_\crit=1-\gamma.
\end{equation}
By abuse of notation we keep the same name for the non-dimensionalized variables and quantities. One can check that the dimensionless vorticity is now $\omega=\gamma$, and for the rest of the paper, we will only use the notation $\gamma$. 

In these non-dimensional variables \eqref{harmonic zeta dim}--\eqref{bottom eta dim} become
\begin{subequations}\label{fullproblem}
	\begin{alignat}{2}
	\label{harmonic eta}
	\Delta\eta&=0&\qquad&\text{in }\mathcal{R},\\
	\label{harmonic zeta}
	\Delta\zeta&=0&\qquad&\text{in }\mathcal{R},\\ 
	\label{kinematic}
	\zeta&=-\tfrac{1}{2}\gamma+1-\tfrac{1}{2}\gamma \eta^2&\qquad&\text{on }\Gamma,\\
	\label{dynamic}
	(\zeta_y+\gamma \eta\eta_y)^2&=(1-2\alpha (\eta-1))|\nabla\eta|^2&\qquad&\text{on }\Gamma,\\
	\label{bottom eta}
	\eta&=0&\qquad&\text{on }B,\\
	\label{bottom zeta}
	\zeta&=0&\qquad&\text{on }B.
	\end{alignat}
	We additionally require the regularity
	\begin{gather}
	\label{regularity}
	\eta,\zeta \in C_{\mathrm{b}}^{3+\beta}(\overline{\mathcal{R}}),
	\end{gather}
	and the symmetry
	\begin{equation}\label{symmetry}
	\text{$\eta$ and $\zeta$ are even in $x$}.
	\end{equation}
	
	Because of our non-dimensionalization, 
	\begin{equation*}
		\eta=y\quad\text{and}\quad\zeta=(1-\gamma)y
	\end{equation*} 
	solve \eqref{harmonic eta}--\eqref{regularity} for any $\alpha$. We refer to these as \emph{laminar}, or \emph{trivial} solutions. This motivates looking at the differences
	\begin{equation}\label{tilde}
		w_1:=\eta-y\qquad\text{and}\qquad w_2:=\zeta-(1-\gamma)y.
	\end{equation}
	For convenience we will denote $w=(w_1,w_2)$. We will also strengthen the asymptotic condition \eqref{asymptotic eta dim} so that $w$ and its first and second partials vanish at infinity. That is, we require that
	\begin{equation}\label{tilde regularity}
		w\in C_0^2(\overline{\mathcal{R}}),
	\end{equation}
	with the function space $C_0^2$ defined as in \eqref{C0}. It will be useful in Section~\ref{sect: Functional Analytic}, when working in a functional analytic setting, to reformulate \eqref{fullproblem} in terms of $w$. However, for all qualitative results we will work with the formulation in terms of $\eta$ and $\zeta$.

Finally, we assume that 
\begin{equation}\label{no stagnation on free surface}
\inf_{\mathcal{R}}(1-2\alpha (\eta-1))^2|\nabla\eta|^2>0.
\end{equation}
The first factor not vanishing implies that we cannot have a wave of greatest height, and the second one being nonzero ensures that $\eta$ defines the imaginary part of a conformal mapping. Moreover, one can check that \eqref{no stagnation on free surface} holds whenever $\inf_\mathcal{S}|\nabla\Psi|>0$, in other words, whenever there are no stagnation points on the free surface.  We will see in Section~\ref{sect: Functional Analytic} that \eqref{no stagnation on free surface} is related to the so-called \emph{Lopatinskii constant} of a suitable linearized elliptic problem. 
\end{subequations}

\subsection{Velocity field in conformal variables}
In Sections~\ref{sect: Flow Force} and \ref{sect: Nodal Analysis} we will work with the velocity field of the fluid, and so it is useful to have notation for the velocity components $U$ and $V$ as functions of the conformal variables $x$ and $y$:
\begin{equation*}
	u(x,y):=U(\xi(x,y),\eta(x,y))\quad\text{and}\quad v(x,y):=V(\xi(x,y),\eta(x,y)).
\end{equation*}
Using the chain rule, \eqref{stream function Psi}, \eqref{conformal streamfunction} and \eqref{zeta dim} we get
\begin{equation}\label{conformal velocity}
	(u,v)=\bigg(\frac{\eta_x\zeta_x+\eta_y\zeta_y}{\eta^2_x+\eta^2_y}+\gamma \eta,\frac{\eta_x\zeta_y-\eta_y\zeta_x}{\eta^2_x+\eta^2_y} \bigg).
\end{equation}
Working with \eqref{conformal velocity} directly often becomes tedious and so we collect properties of $u$ and $v$ here to make further calculation throughout the paper more transparent. 

To begin with, notice that both components of the velocity field are harmonic functions in $\mathcal{R}$ since 
\begin{equation*}
u-iv=\frac{\zeta_y+i\zeta_x}{\eta_y+i\eta_x}+\gamma\eta,
\end{equation*}
where the right hand side is holomorphic. Moreover, \eqref{no stagnation on free surface} ensures that the denominator in \eqref{conformal velocity} never vanishes. It can now easily be checked that $u$ and $v$ solve
\begin{subequations}\label{velocity problem}
	\begin{alignat}{2}
		\label{invariant 1}
		u_x+v_y&=\gamma\eta_x&\qquad&\text{in }\mathcal{R},\\
		\label{invariant 2}
		u_y-v_x&=\gamma\eta_y&\qquad&\text{in }\mathcal{R},\\
		\label{kinematic velocity}
		u\eta_x-v\eta_y&=0&\qquad&\text{on }\Gamma,\\
		\label{dynamic velocity}
		u^2+v^2+2\alpha(\eta-1)&=1&\qquad&\text{on }\Gamma,\\
		\label{kinematic velocity bottom}
		v&=0&\qquad&\text{on }B,
	\end{alignat}
	where \eqref{invariant 1} and \eqref{invariant 2} are due to the incompressibility condition and the vorticity equation, respectively. Plugging \eqref{tilde regularity} into \eqref{conformal velocity} yields the asymptotic conditions
	\begin{equation}\label{asymptotic velocity}
		\lim_{x\to\pm\infty}u=(1-\gamma)+\gamma y,\qquad\lim_{x\to\pm\infty}v=0.
	\end{equation}
	Finally, \eqref{regularity} yields the regularity
	\begin{equation}
		u,v\in C_\be^{2+\beta}(\overline{\mathcal{R}}),
	\end{equation}
	and from \eqref{symmetry} we see that
	\begin{equation}
		u\text{ is even in }x\quad\text{ and }\quad v\text{ is odd in }x.
	\end{equation}
\end{subequations}
\begin{remark}
	It is interesting to note that \eqref{velocity problem} combined with \eqref{harmonic eta} and \eqref{bottom eta} provides us with a closed system for $u,v$ and $\eta$ in $\mathcal{R}$.
\end{remark}

\section{Flow force}\label{sect: Flow Force}
In this section, we will work with the flow force, an invariant for steady waves. Its definition is motivated by the divergence form of the horizontal component \eqref{horizontal momentum} of the momentum equation, 
\begin{equation}\label{divergence form of Euler}
	\big(P-P_{\text{atm}}+U^2\big)_X+\big(VU\big)_Y=0.
\end{equation}
Traditionally, the flow force is defined as  
\begin{equation}\label{flowforce physical usually}
\int\big(P-P_{\text{atm}}+U^2\big)\,dY,
\end{equation}
where the integral is taken over a vertical cross section of the fluid. By differentiating \eqref{flowforce physical usually} with respect to $X$ and using the boundary conditions \eqref{kinematic top}--\eqref{pressure}, one can show that the integral is independent of $X$. 

However, our case is slightly more complicated as we are studying waves which can have overhanging profiles. Therefore, it is more convenient to integrate along lines of constant $x$ and define the flow force as
\begin{equation}\label{flowforce physical}
S=\int_{x=\text{constant}}\big(P-P_{\text{atm}}+U^2\big)\,dY-UV\,dX.
\end{equation} 
where the second term in the integrand comes from the second term in \eqref{divergence form of Euler}. We will need the flow force in two different instances: in Section~\ref{subsect: Nonexistence of Bores} to eliminate the existence of bores and in Section~\ref{subsect: Froude number} to provide upper bounds for the wave speed parameter $\alpha$, proving Theorem~\ref{thm: dim froude}. We begin by collecting certain useful facts about this quantity in the following lemma.
\begin{lemma}\label{lem: flowforce indep of x}
  Assume $(\eta,\zeta,\alpha)$ solves \eqref{harmonic eta}--\eqref{regularity}. Then the flow force \eqref{flowforce physical} is given by
	\begin{equation}\label{flowforce}
\begin{aligned}
	S(x;\eta,\zeta,\alpha):=\frac{1}{2}\int_{0}^{1}\frac{\eta_y(\zeta_y^2-\zeta_x^2)+2\eta_x\zeta_x\zeta_y}{\eta_x^2+\eta_y^2}\,dy
	-\bigg(\frac{\gamma^2}{6}\eta^3+\frac{\alpha}{2}\eta^2-\frac{2\alpha+1}{2}\eta\bigg) \bigg|_{y=1},
\end{aligned}
	\end{equation}
	and is independent of $x$.
\end{lemma}
\begin{proof}
	In our dimensionless variables \eqref{bernoulli} takes the form
	\begin{equation}\label{bernoulli dimless}
		P+\tfrac{1}{2}(U^2+V^2)+\alpha(\eta-1)-\gamma\Psi=1+P_\text{atm}-\gamma\big(1-\tfrac{1}{2}\gamma\big),
	\end{equation}
	where here we have explicitly provided the constant term on the right hand side.	Using \eqref{bernoulli dimless} expressed in terms of the conformal variables, \eqref{flowforce physical} becomes
	\begin{equation}\label{flow force messy}
	\begin{aligned}
	S(x;\eta,\zeta,\alpha)=\int_{0}^{1}\Big(\tfrac{1}{2}(u^2-v^2)-\alpha(\eta-1)+\tfrac{1}{2}+\gamma\big(\zeta+\tfrac{1}{2}\gamma\eta^2-1+\tfrac{1}{2}\gamma \big) \eta_y+uv\eta_x\Big)\,dy.
	\end{aligned}
	\end{equation}
	A tedious yet straightforward calculation shows that 
	\begin{equation}\label{flow force tedious}
	\begin{aligned}
	\tfrac{1}{2}\big(u^2-v^2+2\gamma\zeta\big)\eta_y+uv\eta_x=\Big[\gamma \eta\big(\zeta+\tfrac{1}{6}\gamma \eta^2\big)\Big]_y+\frac{1}{2}\frac{\eta_y(\zeta_y^2-\zeta_x^2)+2\eta_x\zeta_x\zeta_y}{\eta_x^2+\eta_y^2}.
	\end{aligned}
	\end{equation}
  Inserting \eqref{flow force tedious} into \eqref{flow force messy}, integrating several total derivatives explicitly and simplifying the resulting boundary terms using the kinematic boundary condition \eqref{kinematic} yields \eqref{flowforce}.
	
  In order to show that $S$ is invariant, we first observe that the integrand in \eqref{flowforce} is the real part of the holomorphic function
  \begin{equation}\label{holomorphic}
    \frac{(\zeta_y+i\zeta_x)^2}{\eta_y+i\eta_x}
    =
    \frac{\eta_y(\zeta_y^2-\zeta_x^2)+2\eta_x\zeta_x\zeta_y}{\eta_x^2+\eta_y^2}
    +i\,\frac{2\eta_y\zeta_x\zeta_y-\eta_x(\zeta_y^2-\zeta_x^2)}{\eta_x^2+\eta_y^2}.
  \end{equation}
  Differentiating under the integral and using the Cauchy--Riemann equations we therefore obtain
  \begin{align}
    \label{has a rhs}
    \frac d{dx} \int_0^1 
    \frac 12 \frac{\eta_y(\zeta_y^2-\zeta_x^2)+2\eta_x\zeta_x\zeta_y}{\eta_x^2+\eta_y^2}
    \, dy
    = 
    \frac 12 \frac{2\eta_y\zeta_x\zeta_y-\eta_x(\zeta_y^2-\zeta_x^2)}{\eta_x^2+\eta_y^2}
    \bigg|_{y=1},
  \end{align}
  where the boundary term at $y=0$ vanishes thanks to \eqref{bottom eta} and \eqref{bottom zeta}. Differentiating the kinematic boundary condition \eqref{kinematic}, we can replace $\zeta_x$ with $-\gamma \eta\eta_x$ on the right hand side of \eqref{has a rhs}. Completing the square in $\zeta_y$, we can then completely eliminate $\zeta$ by using the dynamic boundary condition \eqref{dynamic}. Factoring, we obtain 
  \begin{align}
    \label{holomorphic surface stuff}
    \frac 12 \frac{2\eta_y\zeta_x\zeta_y-\eta_x(\zeta_y^2-\zeta_x^2)}{\eta_x^2+\eta_y^2}
    = \bigg(\frac \gamma 2 \eta^2 + \alpha \eta - \frac{2\alpha +1}2 \bigg) \eta_x 
    \quad \text{ on } \Gamma,
  \end{align}
  and the result follows.
\end{proof}

\subsection{Nonexistence of bores and compactness}\label{subsect: Nonexistence of Bores}
Since we work in an unbounded domain, we no longer have compact embeddings between Hölder spaces. However, we will see that, for monotone waves, the only way to lose compactness is through the existence of a bore. The aim of the first part of this subsection is to rule out the latter possibility. 

A \emph{bore} is a solution $(\eta,\zeta,\alpha)$ which solves \eqref{harmonic eta}--\eqref{regularity} and satisfies  
\begin{equation}\label{definition of bore precis}
	\lim_{x\to\pm\infty}(\eta,\zeta)(x,y)=(\eta_{\pm}(y),\zeta_{\pm}(y)),
\end{equation}
with distinct limits $(\eta_-,\zeta_-)$ and $(\eta_+,\zeta_+)$. By a translation argument these limits must also solve \eqref{harmonic eta}--\eqref{regularity} and  hence are of the form
\begin{equation*}
	\eta_{\pm}(y)=\hat{\eta}_\tr(y;d_\pm),\qquad\zeta_\pm(y)=\hat{\zeta}_\tr(y;d_\pm),
\end{equation*}
with $d_-\neq d_+$ where
\begin{equation}\label{parameterized trivial solutions}
\begin{aligned}
&\hat{\eta}_\tr(y;d):=dy,\\
&\hat{\zeta}_\tr(y;d):=\bigg(\frac{2-\gamma}{2d}-\frac{\gamma d}{2}\bigg)dy.
\end{aligned}
\end{equation}
Notice that solutions to \eqref{fullproblem} satisfy \eqref{definition of bore precis}--\eqref{parameterized trivial solutions} with $d_+=d_-=1$. Since \eqref{parameterized trivial solutions} also solves the dynamic boundary condition \eqref{dynamic}, we obtain
\begin{equation}\label{parameterized Q equality}
	\hat{Q}(d_-)=\hat{Q}(d_+)=\hat{Q}(1),
\end{equation}
where
\begin{equation}\label{Q}
\hat{Q}(d)=\frac{1}{d^2}\bigg(\frac{2-\gamma}{2}+\frac{\gamma d^2}{2} \bigg)^2+2\alpha(d-1). 
\end{equation}
Moreover, defining
\begin{equation}\label{parameterized flowforce trivial}
\hat{S}(d):=S\big(x;(\hat{\eta}_\tr(\cdot\,;d),\hat{\zeta}_\tr(\cdot\,;d) )\big),
\end{equation}
the invariance of the flow force given in Lemma~\ref{lem: flowforce indep of x} implies that
\begin{equation}\label{conjugate flows}
\hat{S}(d_-)=\hat{S}(d_+)=\hat{S}(1).
\end{equation}
Together, \eqref{parameterized Q equality} and \eqref{conjugate flows} are the so-called \emph{conjugate flow equations} for our problem \cite{benjamin:impulse}. 
 
The following two lemmas collect properties of the parameterized forms \eqref{Q} and \eqref{parameterized flowforce trivial} of the Bernoulli constant and the flow force. We point out that, in the absence of critical layers, analogues of these lemmas have been shown in \cite{pressure}; also see \cite{strat}. 

\begin{lemma}\label{lem:bernoulli}
	The function $\hat{Q}$ in \eqref{Q} is a strictly convex function of the asymptotic depth $d>0$.  In particular, it admits a unique minimum, $d=d_\crit$. Moreover, there exists a unique $d_*$ with $\hat{Q}(d_*)=\hat{Q}(1)$. For $\alpha<\alpha_\crit$ we have $d_*\in(d_\crit,\infty)$, while for $\alpha>\alpha_\crit$ we have $d_*\in(0,d_\crit)$.
\end{lemma}
\begin{proof}
	Differentiating \eqref{Q} twice with respect to $d$, we find
	\begin{align*}
    \hat{Q}'(1) = 2(\alpha-\alpha_\crit), 
    \qquad 
    \hat{Q}''(d)=\frac{1}{6d^4}\bigg(1-\frac{\gamma}{2} \bigg)^2+\frac{\gamma^2}{2}>0.
	\end{align*}
	The statement then follows from the fact that $\hat{Q}\to\infty$ both as $d\to\infty$ and $d\to0$.
\end{proof}

\begin{lemma}\label{lem:Sprime}
  The function $\hat{S}(d)$ defined in \eqref{parameterized flowforce trivial} satisfies
	\begin{equation}\label{Sprime}
	\hat{S}'(d)=\frac{1}{2}\big(\hat{Q}(1)-\hat{Q}(d)\big).
	\end{equation}
	In particular, by the convexity of $\hat{Q}$, we have $\hat{S}(d_*)>\hat{S}(1)$ for $\alpha<\alpha_\crit$ and $\hat{S}(d_*)<\hat{S}(1)$ for $\alpha>\alpha_\crit$.
\end{lemma}
\begin{proof}
	We can calculate $\hat{S}(d)$ explicitly:
	\begin{equation*}		\hat{S}(d)=\frac{(2-\gamma)^2}{8d}-\frac{\gamma^2d^3}{24}-\frac{(2-\gamma)\gamma d}{4}+d\alpha-\frac{d^2\alpha}{2}+\frac{\hat{Q}(1)}{2}d
	\end{equation*}
	from which we easily get \eqref{Sprime}. The last part of the statement follows from Lemma~\ref{lem:bernoulli}. Indeed, assuming without loss of generality that $\alpha<\alpha_\crit$, since $Q(d)<Q(1)$ for $1<d<d_*$, we have
	\begin{equation*}
		S(d_*)-S(1)=\int_{1}^{d_*}\big(Q(1)-Q(s)\big)\,ds>0.\qedhere
	\end{equation*}
\end{proof}

We are now ready to prove the following result.
\begin{lemma}[Nonexistence of bores]\label{lem:nobores}
	The conjugate flow equations \eqref{parameterized Q equality} and \eqref{conjugate flows} have no solutions other than $d=1$. In particular, \eqref{harmonic eta}--\eqref{regularity} does not admit bore solutions as defined in \eqref{definition of bore precis}--\eqref{parameterized trivial solutions}.
\end{lemma}
\begin{proof}
	Let us assume that $(\eta,\zeta)$ is a bore solution to \eqref{harmonic eta}--\eqref{regularity}. Then \eqref{conjugate flows} must hold. If $\alpha=\alpha_\crit$, then arguing as in Lemmas~\ref{lem:bernoulli} and \ref{lem:Sprime}, $d_*=d_-=d_+=1$. Otherwise, \eqref{parameterized Q equality} and Lemma~\ref{lem:bernoulli} imply that $d_\pm\in\{1,d_*\}$, while Lemma~\ref{lem:Sprime} implies $d_-\neq d_*$ and $d_+\neq d_*$. We must therefore have $d_-=d_+=1$.  
\end{proof}

We are now ready to prove compactness. 
\begin{lemma}[Compactness]\label{lem: compactness}
	Let $(\eta_n,\zeta_n,\alpha_n)$ be a sequence of solutions to \eqref{fullproblem}. If
	\begin{equation}\label{uniform bounds}
	\sup_n\|(\eta_n,\zeta_n)\|_{C^{3+ \beta}(\mathcal{R})}<\infty\qquad\text{and}\qquad \inf_n\inf_{\mathcal{R}}(1-2\alpha_n(\eta_{n}-1))^2|\nabla\eta_{n}|^2>0,	\end{equation}
	as well as 
	\begin{equation}\label{monotonicity}
	\partial_x\eta_n\leq0\textup{ for }x\geq0
	\end{equation}
	hold, then we can extract a subsequence with $(\eta_n,\zeta_n)\to(\eta,\zeta)$ in $C_{\mathrm{b}}^{3+\beta}(\overline{\mathcal{R}})$. 
\end{lemma}
\begin{proof}
	Note that we only give monotonicity assumptions on $\eta$, since from \eqref{kinematic}, we see that \eqref{monotonicity} and a maximum principle argument show that $-\gamma\partial_x\zeta_n\leq0\text{ for }x\geq0$. Moreover, we can assume without loss of generality that $\alpha_n\to\alpha\in\mathbb{R}$. We argue as in \cite[Lemma 6.3]{strat}. 
	
	We begin by assuming that
	\begin{equation}\label{uniform decay}
		\lim_{x\to\infty}\sup_n\sup_y\big|(\eta_n,\zeta_n)(x,y)-\big(y,(1-\gamma)y \big)\big|=0
	\end{equation}
	holds. Then, using Arzelà--Ascoli, we can extract subsequences so that $(\eta_n,\zeta_n)\to(\eta,\zeta)$ in $C_\text{loc}^{3}(\overline{\mathcal{R}})$ and by \eqref{uniform decay} also in $L^\infty(\overline{\mathcal{R}})$, for some $(\eta,\zeta)$ which solve \eqref{fullproblem}. Therefore, the differences 
	\begin{equation*}\label{vn}
	v_n^{(1)}=\eta_n-\eta\qquad\text{and}\qquad v_n^{(2)}=\zeta_n-\zeta
	\end{equation*}
	satisfy
	\begin{equation}\label{sequence no bore}
		\|(v_n^{(1)},v_n^{(2)})\|_{L^{\infty}(\mathcal{R})}\to0\qquad\text{as }n\to\infty,
	\end{equation}
	after extraction. It remains to show that $(\eta_n,\zeta_n)\to(\eta,\zeta)$ in $C_\be^{3+\beta}(\overline{\mathcal{R}})$. 
	
	On the surface $\Gamma$, $v_n^{(1)}$ and $v_n^{(2)}$ solve the linear boundary conditions 
	\begin{equation}\label{vn problem}
	\begin{aligned}
	a_{22}\partial_yv_n^{(2)}-a_{11}\partial_xv_n^{(1)}-a_{12}\partial_yv_n^{(1)}-b_{1}v_n^{(1)}&=f_n,\\
	c_1v_n^{(1)}+v_n^{(2)}&=0,
	\end{aligned}
	\end{equation}
	where the coefficients and $f_n$ are given by
	\begin{equation*}\label{vn coefficients}
	\begin{aligned}
	a_{11}&=\big(2\alpha_n(\eta_{n}-1)+2\alpha(\eta-1)-2\big)\eta_x,\\
	a_{12}&=\big(2\alpha_n(\eta_{n}-1)+2\alpha(\eta-1)-2\big)\eta_y-a_{22}c_1,\\	
	a_{22}&=\zeta_{ny}+\gamma \eta_n\eta_{ny}+\zeta_y+\gamma \eta\eta_y,\\
	b_1&=\tfrac\gamma 2 (\eta_{ny}+\eta_y+\eta+\eta_n)a_{22},\\
	c_1&=\tfrac{1}{2}\gamma( \eta+\eta_n),\\
	f_n&=\big((\eta_{nx}-\eta_x)^2+(\eta_{ny}-\eta_y)^2\big)\big(2\alpha_n(\eta_n-1)-1 \big)\\&\qquad+2\big((\eta_x\eta_{nx}+\eta_y\eta_{ny}) \big)\big(\alpha_n(\eta_n-1)-\alpha(\eta-1)\big)+(\eta_{ny}-\eta_y)(\eta_n+\eta)a_{22}.
	\end{aligned}
	\end{equation*}
	From \eqref{uniform bounds}, all coefficients and $f_n$ are uniformly bounded in $C_\be^{2+\beta}(\Gamma)$. Applying Lemma~\ref{lem: schauder estimates}, \eqref{vn problem} implies a Schauder estimate provided the uniform bound
	\begin{equation*}\label{vn boundary bounds}
		\big(2\alpha_n(\eta_{n}-1)+2\alpha(\eta-1)-2\big)^2(\eta_x^2+\eta_y^2)\geq\delta,
	\end{equation*}
	is satisfied on $\Gamma$, for some $\delta>0$. By assumption \eqref{uniform bounds}, this inequality holds. From \eqref{tilde regularity} and \eqref{sequence no bore}, we therefore get
	\begin{equation*}
	\|(v_n^{(1)},v_n^{(2)})\|_{C^{3+\beta}(\mathcal{R})}\leq C\big(\|f_n\|_{C^{2+\beta}(\Gamma)}+ 	\|(v_n^{(1)},v_n^{(2)})\|_{L^{\infty}(\mathcal{R})}\big)\to0,
	\end{equation*}
	as $n\to\infty$. Thus, $(\eta_n,\zeta_n)\to(\eta,\zeta)$ in $C_{\mathrm{b}}^{3+\beta}(\overline{\mathcal{R}})$.
	
	Let us now assume that \eqref{uniform decay} does not hold. Then we can find a sequence $\{(x_n,y_n)\}\subset\mathbb{R}^2$ with $x_n\to\infty$ and $\ep>0$ so that 
	\begin{equation*}
		\big|(\eta_n,\zeta_n)(x_n,y_n)-\big(y_n,(1-\gamma)y_n \big)\big|\geq\ep,
	\end{equation*}
	for all $n$. Using a translation argument and the monotonicity assumption \eqref{monotonicity} exactly as in Lemma 6.3 in \cite{strat}, the sequence of solutions $(\zeta_n,\eta_n)$ must converge to a bore solution of \eqref{harmonic eta}--\eqref{regularity} as $n\to\infty$. This is a contradiction to Lemma~\ref{lem:nobores}, and therefore \eqref{uniform decay} must hold.
\end{proof}

\subsection{Bounds on the Froude number}\label{subsect: Froude number}
The goal of this subsection is to provide a lower bound on the Froude number. This result will rule out the alternative that solutions along the global continuum will reach the critical Froude number value \eqref{alpha crit} at which our linearized operators cease to be Fredholm. Due to the possibility of having internal stagnation points and overhanging wave profiles, we cannot work with the Dubreil-Jacotin formulation and therefore the approach used in \cite{froude} cannot be applied here. Instead we develop a new argument involving the function $\Phi$ defined in \eqref{flowforce flux function} below, related to the flow force flux function derived in \cite{klw:subcritical}.  We begin by showing that this function satisfies an elliptic equation. Compared to the function in \cite{klw:subcritical}, ours is harmonic but has more complicated boundary conditions. Specifically, we obtain an equation on the surface $\Gamma$ which is related to the Babenko formulation of our problem \eqref{fullproblem}. We then obtain the desired bounds by using this boundary condition to establish the integral identity \eqref{froudenumberboundintegral}. 
   
We begin by defining our variant of the flow force flux function:
	\begin{equation}\label{flowforce flux function}
\Phi(x,y):=\int_{0}^{y}\bigg(\frac{\eta_y(\zeta_y^2-\zeta_x^2)+2\eta_x\zeta_x\zeta_y}{\eta_x^2+\eta_y^2}
    +(1-\gamma^2)\eta_y+2(\gamma-1)\bigg)\,dy.
	\end{equation}
From \eqref{tilde regularity}, we clearly see that $\Phi\to0$ as $|x|\to\infty$. Notice that the first term in the integrand is the same as the integrand in \eqref{flowforce}, except that here the upper limit of the integral is free rather than being fixed at $y=1$.

\begin{proposition}\label{prop: phi}
	Suppose that $(\eta,\zeta,\alpha)$ solve \eqref{fullproblem}. Then the function $\Phi$ defined in \eqref{flowforce flux function} solves the Dirichlet problem
	\begin{subequations}\label{flowforceproblem}
	\begin{align}
	\label{Phi harmonic}
	\Delta\Phi&=0&\qquad&\textup{in }\mathcal{R},\\
	\label{Phi surface identity}
	\Phi&=(\alpha+\gamma^2)(\eta-1)^2+\tfrac{1}{3}\gamma^2(\eta-1)^3&\qquad&\textup{on }\Gamma,\\
	\label{Phi bottom}
	\Phi&=0&\qquad&\textup{on }B.
	\end{align}
	\end{subequations}
\end{proposition}
\begin{proof}
  Clearly, from the definition \eqref{flowforce flux function} of $\Phi$, \eqref{Phi bottom} holds. The fact that $\Phi$ is harmonic follows from the fact that the first term in the integrand is the real part of the holomorphic function \eqref{holomorphic}. Setting \eqref{flowforce} equal to its limit as $|x|\to\infty$, subtracting it from \eqref{flowforce} and pulling the zeroth and first order terms into the integral yields \eqref{Phi surface identity}. 
\end{proof}

Using $\Phi$ and Proposition~\ref{prop: phi}, we are now able to establish a useful integral identity.

\begin{proposition}\label{prop: integral identity}
	Any solution $(\eta,\zeta,\alpha)$ to \eqref{fullproblem} satisfies the integral identity
	\begin{equation}\label{froudenumberboundintegral}
    \begin{aligned}
      (1-\alpha-\gamma)\int_{-M}^{M}w_1\,dx
      =\alpha\int_{-M}^{M}w_1w_{1y}\,\,dx
      +\frac{\alpha+\gamma^2}{2}\int_{-M}^{M}w_1^2\,dx
      +\frac{\gamma^2}{6}\int_{-M}^{M}w_1^3\,dx
      +o(1)
    \end{aligned}
  \end{equation}
	as $M\to\infty$, where the integrals are over the surface $y=1$ and $w_1(x,y)=\eta(x,y)-y$ was defined in \eqref{tilde}.
\end{proposition}

\begin{proof}

Multiplying \eqref{Phi harmonic} by $y$ and integrating by parts twice, we get, for any $M>0$, 
\begin{equation*}
\begin{aligned}
0&=-\int_{-M}^{M}\int_{0}^{1}\Delta\Phi\cdot y\,dydx\\
&=\int_{-M}^{M}\Phi\,dx\bigg|_{y=0}^{y=1}-\int_{-M}^{M}\Phi_{y} \cdot y\,dx\bigg|_{y=0}^{y=1}-\int_{0}^1\Phi_x\cdot y\,dy\bigg|_{x=-M}^{x=M},
\end{aligned}
\end{equation*}
which yields
\begin{equation}\label{full integral Phi-Phi_Y}
\int_{-M}^{M}(\Phi-\Phi_{y} y)\,dx\bigg|_{y=0}^{y=1}=\int_{0}^1\Phi_x\cdot y\,dy\bigg|_{x=-M}^{x=M}=o(1)\text{ as }M\to\infty.
\end{equation}
The last equality is true since $\Phi_x\rightarrow0$ as $x\rightarrow\infty$ by \eqref{tilde regularity}. Let us now look more closely at the left-hand side. 

  Notice that the second term in the first integrand of \eqref{full integral Phi-Phi_Y}, $\Phi_y\cdot y$, vanishes on $B$. On the other hand, arguing exactly as above \eqref{holomorphic surface stuff}, we find
  \begin{align}
    \label{holomorphic surface again}
    \begin{aligned}
      \Phi_y 
      &= \frac{\eta_y(\zeta_y^2-\zeta_x^2)+2\eta_x\zeta_x\zeta_y}{\eta_x^2+\eta_y^2}
      + (1-\gamma^2)\eta_y+2(\gamma-1)\\
      &= 
      -2\gamma\eta\zeta_y
      +(2+2\alpha-\gamma^2)\eta_y
      -2\alpha \eta\eta_y
      -2\gamma^2 \eta^2 \eta_y
      +2(\gamma-1)
    \end{aligned}
    \quad \text{ on } \Gamma.
  \end{align}
  Since $\eta_x$ and $\zeta_x$ vanish at infinity, the Gauss--Green theorem and the kinematic boundary condition \eqref{kinematic} give
  \begin{align}
    \label{gauss green}
    \begin{aligned}
      \int_{-M}^{M}\eta\zeta_y\,dx
      &=\int_{-M}^{M}\eta_y \zeta\,dx+o(1)\\
      &=\int_{-M}^{M}\eta_y (-\tfrac 12 + \gamma + 1 - \tfrac 12 \gamma \eta^2)\,dx+o(1)
      \quad\text{as }M\to\infty.
    \end{aligned}
  \end{align}
  Combining \eqref{holomorphic surface again} and \eqref{gauss green}, and rewriting in terms of the unknown $w_1=\eta-y$ defined in \eqref{tilde}, we find
  \begin{equation*}
    \label{integral Phi_y 1}
    \int_{-M}^M \Phi_y \cdot y\, dx \bigg|^{y=1}_{y=0}
    = 
    \int_{-M}^M 
    2\Big(
    (1-\gamma) w_{1y}
    - \alpha w_1 - \alpha w_1 w_{1y}
    \Big)
    \, dx \bigg|_{y=1}
    + o(1)
    \qquad 
    \text{ as } M \to \infty.
  \end{equation*}
  Applying the Gauss--Green theorem to $w_1$ and $y$, we also have
  \begin{equation*}
    \int_{-M}^{M}w_{1y}\,dx=\int_{-M}^{M}w_1\,dx+o(1)\quad\text{as }M\to\infty.
  \end{equation*}
  Using \eqref{Phi surface identity} therefore gives
  \begin{equation*}
    \int_{-M}^{M}(\Phi-\Phi_{y}\cdot y)\,dx\bigg|_{y=0}^{y=1}
      =2\int_{-M}^{M}\Big((\gamma+\alpha-1)w_1+\frac{\alpha+\gamma^2}2 w_1^2+\frac{\gamma^2}{6}w_1^3+\alpha w_1w_{1y}\Big)\,dx+o(1),
  \end{equation*}
as $M\to\infty$, so that \eqref{full integral Phi-Phi_Y} rearranges to the desired identity \eqref{froudenumberboundintegral}.
\end{proof}

We now use the identity \eqref{froudenumberboundintegral} to give an upper bound on $\alpha$, proving Theorem~\ref{thm: dim froude}. We will see in Section~\ref{sect: Existence Results} that this upper bound is in fact sharp.
\begin{theorem}\label{thm: bound on alpha}
	Let $(\eta,\zeta,\alpha)$ solve \eqref{fullproblem} with $\alpha>0$. If $\eta\geq 1$ on the surface $\Gamma$, then either
	\begin{equation}\label{bound on alpha}
	\alpha<\alpha_\crit=1-\gamma
	\end{equation} 
	holds, or else the solution is trivial. Here $\alpha_\crit$ is defined as in \eqref{alpha crit}.
\end{theorem}
\begin{proof}
	Suppose that $w_1\not\equiv0$. It is easy to check that as a consequence $w_1$ must be strictly positive somewhere on the surface $\Gamma$. 
	
	We claim that all the integrals in \eqref{froudenumberboundintegral} are positive for $M$ sufficiently large. We only need to check the first integral on the right-hand side, with the integrand $w_1w_{1y}$, since the claim holds for the others by assumption. This is easily done by noticing that 
	\begin{equation*}
	0<\int_{-M}^M\int_{0}^1|\nabla w_1|^2\,dydx=\int_{-M}^M\int_{0}^1\nabla\cdot(w_1\nabla w_1)\,dydx=\int_{-M}^Mw_1w_{1y}\,dx\bigg|_{y=1}+\int_{0}^1w_1w_{1x}\,dy\bigg|_{-M}^{M}.
	\end{equation*}  
	Since the second term on the right-hand side goes to zero as $|x|\to\pm\infty$, the claim has been proven. 
	
	Since the coefficients of the first three terms on the right-hand side of \eqref{froudenumberboundintegral} are strictly positive, \eqref{bound on alpha} must hold. 
\end{proof}

	\section{Nodal analysis}\label{sect: Nodal Analysis}
In this section, we will prove that the solutions to \eqref{fullproblem} satisfy a monotonicity property which is traditionally referred to as \emph{nodal property}. More specifically, we wish to show that the vertical component $\eta$ of the parametrization of the wave profile $\mathcal{S}$ as defined in \eqref{surface parametrization} is strictly decreasing on either side of the wave crest, that is, 
\begin{equation}\label{nodal eta}
	\eta_x<0\quad\text{in }(\mathcal{R}\cup\Gamma)\cap\{x>0\}\qquad\text{and}\qquad\eta_x>0\quad\text{in }(\mathcal{R}\cup\Gamma)\cap\{x<0\}.
\end{equation}
As we will see in Section~\ref{subsect: small-amplitude}, \eqref{nodal eta} is motivated by the structure of the small-amplitude solutions in a neighborhood of the bifurcation point. In the periodic case, the main purpose of this property is to rule out a loop alternative for the global curve of solutions. For solitary waves, monotonicity has an even bigger importance in the sense that its validity, combined with the nonexistence of bores, ensures compactness in the unbounded domain. We remark that the monotonicity property in \eqref{nodal eta} does not rule out the existence of waves with overhanging profiles; see the left hand side of Figure~\ref{fig:nodal}. 

\begin{figure}
	\centering
	\includegraphics[scale=1.1]{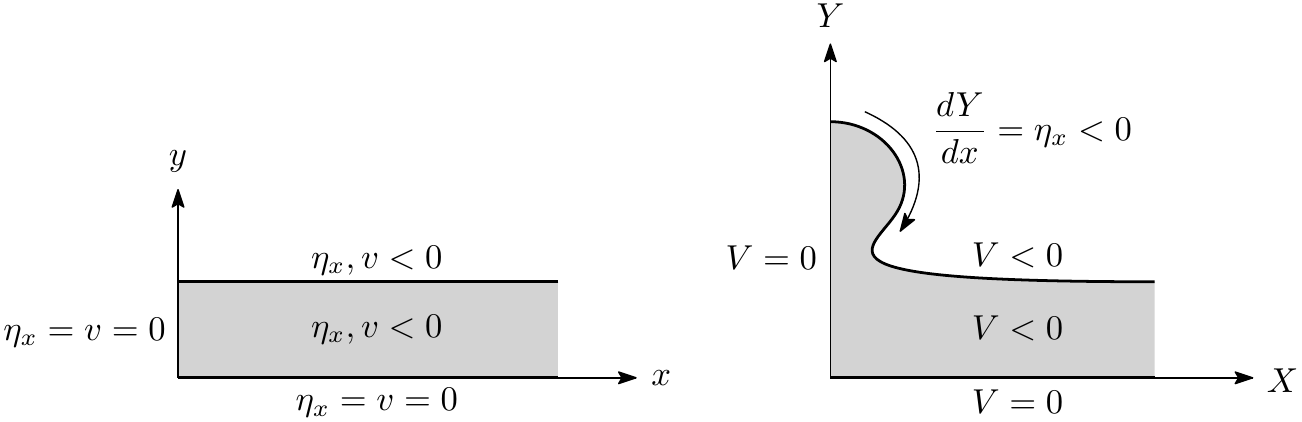}
	\caption{The nodal properties \eqref{nodal eta} and \eqref{v<0} in the right half of the domain. Note that the monotonicity property $\eta_x<0$ does not prevent the wave from having an overhanging profile in the physical variables.
		\label{fig:nodal}
	}
\end{figure}

The aim is to prove that monotonicity is conserved along the global curve of solutions $\mathscr{C}$ by showing that \eqref{nodal eta} defines both a relatively open and closed subset of the set of nontrivial solutions of \eqref{fullproblem}. For the closed condition, we closely follow the idea developed in \cite{csv:critical} in which the main thrust of the argument concerns the vertical component $v$ of the velocity field, as defined in \eqref{conformal velocity}. Rather than working with $\eta,\zeta$ and $v$, we find it more convenient to work with $u,v$ and $\eta$ (see Lemma~\ref{lem: v problem} and \eqref{v problem} below). This streamlines the presentation somewhat and is especially useful when considering the behavior of the problem at infinity in Lemma~\ref{lem: smallstrict}.

The open condition is more delicate. For periodic waves, such as the ones studied in \cite{cs:exact} and \cite{csv:critical}, one can show that the set consisting of the nodal property \eqref{nodal eta} combined with certain additional monotonicity properties is open in $C_\be^3(\overline{\mathcal{R}})$. Due to the unboundedness of the domain, this is not true when working with solitary waves. To overcome this difficulty we follow the approach in \cite{paper} and \cite{strat} and split the domain into a bounded rectangle and two semi-infinite strips. The open condition in the rectangle then follows as in the periodic setting, and it remains to closely study the qualitative behavior of solutions at infinity in the semi-unbounded strip and glue the domain pieces back together. 

Let us now precisely state the nodal property in terms of the vertical component $v$ of the velocity field, as defined in \eqref{conformal velocity}. Since we are assuming in \eqref{symmetry} that our solutions to \eqref{fullproblem} are symmetric, it suffices to work in the positive half of the domain $\mathcal{R}$. We denote 
\begin{equation*}
\mathcal{R}^+:=\{(x,y)\in\mathcal{R}: x>0\}\qquad\text{and}\qquad\Gamma^+:=\{(x,y)\in\Gamma:x>0 \}.
\end{equation*} 
The monotonicity property in terms of $v$ is
\begin{equation}\label{v<0}
	v<0\qquad\text{in }\Gamma^+\cup\mathcal{R}^+.
\end{equation}
This next result shows that \eqref{v<0} along with its analogue on the negative half of the domain implies \eqref{nodal eta}; see Figure~\ref{fig:nodal}.

\begin{lemma}\label{lem: nodal eta}
	Suppose $(\eta,\zeta,\alpha)$ solves \eqref{fullproblem} and $v$ is defined by \eqref{conformal velocity}. If \eqref{v<0} holds, then $\eta_x<0$ in $\Gamma^+\cup\mathcal{R}^+$.
\end{lemma}
\begin{proof}
	By the kinematic boundary condition \eqref{kinematic velocity}, the vector fields $(u,v)$ and $(\eta_y,\eta_x)=(\xi,\eta)_x$ are non-vanishing and parallel when restricted to $\Gamma$. Moreover, the asymptotic conditions \eqref{tilde regularity} and \eqref{asymptotic velocity} show that $u,n_y\to1$ as $x\to\infty$, ensuring that $v$ and $\eta_x$ have the same sign. Since \eqref{v<0} implies that $\eta_x<0$ on $\Gamma^+$, the result now follows from applying the strong maximum principle to the harmonic function $\eta_x$, which vanishes on the remaining components of the boundary of $\mathcal{R}^+$. 
\end{proof}
The main results in this section are the following two propositions.
\begin{proposition}[Closed condition]\label{prop: closed}
	Suppose that $(\eta,\zeta,\alpha)$ solves \eqref{fullproblem} and define $v$ by \eqref{conformal velocity}. If we have $v\leq0$ on $\Gamma^+$, then the strict inequality \eqref{v<0} holds, or else $v\equiv0$.
\end{proposition}

\begin{proposition}[Open condition]\label{prop: open}
	Fix a solution $(\eta^*,\zeta^*,\alpha^*)$ to \eqref{fullproblem} with $v^*$ from \eqref{conformal velocity} satisfying \eqref{v<0} and $\alpha<1-\gamma$. Then there exists an $\ep>0$ such that for all solutions $(\eta,\zeta,\alpha)$ to \eqref{fullproblem} with  $\|\eta^*-\eta\|_{C^3(\mathcal{R})}+\|\zeta^*-\zeta\|_{C^3(\mathcal{R})}+|\alpha^*-\alpha|<\ep,$ the corresponding vertical velocity $v$ also satisfies \eqref{v<0}.	
\end{proposition}

\subsection{Closed condition}
We begin by showing that the vertical velocity component $v$ as defined in \eqref{conformal velocity} satisfies an elliptic PDE. This will not only enable us to prove the closed condition very easily but, as we will see later on, it will also greatly help organize the arguments for the open condition.

\begin{lemma}\label{lem: v problem}
	Suppose $(\eta,\zeta,\alpha)$ solve \eqref{fullproblem} and $(u,v)$ are defined by \eqref{conformal velocity}. Then $v$ solves the following problem
	\begin{subequations}\label{v problem}
		\begin{alignat}{2}
		\Delta v&=0&\qquad&\textup{in }\mathcal{R},\\
		\label{boundary v}
		u(uv_y-vv_x)-\eta_y(\gamma u+\alpha)v&=0&\qquad&\textup{on }\Gamma,\\
		v&=0&\qquad&\textup{on }B.
		\end{alignat}
	\end{subequations}
\end{lemma} 
\begin{proof}
	Since we know that $(u,v,\eta)$ solve \eqref{velocity problem}, it remains to check \eqref{boundary v}, the boundary condition on $\Gamma$. We begin by differentiating the dynamic boundary condition \eqref{dynamic velocity} with respect to $x$ and multiplying it by $-u/2$ to get 
	\begin{equation}\label{dynamic derivative}
		-u(vv_x+uu_x+\alpha\eta_x)=0.
	\end{equation}
	From \eqref{invariant 1}, we know that $u_x=\eta_x\gamma-v_y$. Inserting this into \eqref{dynamic derivative} and reorganizing yields
	\begin{equation}\label{dynamic derivative 2}
		u(uv_y-vv_x)-\eta_xu(u\gamma+\alpha)=0.
	\end{equation}
	Finally, solving the kinematic condition \eqref{kinematic velocity} for $u\eta_x$ and substituting it in \eqref{dynamic derivative 2} yields \eqref{boundary v}. 
\end{proof}

\begin{proof}[Proof of Proposition~\ref{prop: closed}]
  Assume that $v\not\equiv0$. By \eqref{kinematic velocity bottom} and \eqref{asymptotic velocity}, $v$ vanishes along $B$ and as $x \to \infty$, while the definition \eqref{conformal velocity} together with the evenness of $\eta$ and $\zeta$ implies that $v=0$ on $L$. Since $v \le 0$ on $\Gamma^+$ by assumption, the strong maximum principle therefore implies $v < 0$ in $\mathcal R^+$. We claim that $v < 0$ on $\Gamma^+$. Suppose for the sake of contradiction that $v$ achieves its maximum of $0$ at some point $(x_0,1) \in \Gamma^+$. Then at this point \eqref{boundary v} simplifies to $u^2v_y=0$. The dynamic boundary condition \eqref{dynamic velocity} and \eqref{no stagnation on free surface} imply that $u \ne 0$ at $(x_0,1)$, and so we must have $v_y=0$. But this contradicts the Hopf boundary point lemma. 
\end{proof}

\subsection{Open condition}
As described above, in order to prove the open condition we need to carefully study the behavior of $v$ as $x\to\infty$. To this end, for $M>0$ we split $\mathcal{R}^+$ into a bounded rectangle of length $2M$ and an overlapping semi-infinite strip
\begin{equation*}\label{semi-infinite strip}
\begin{aligned}
\mathcal{R}_M^+:=\{(x,y)\in\mathcal{R}:x>M \},
\end{aligned}
\end{equation*}
with the corresponding analogues for the boundary components
\begin{equation*}
\begin{aligned}
\Gamma_M^+:=\{(x,y)\in\Gamma:x>M\},\qquad
B_M^+:=\{(x,y)\in B: x>M \},\qquad
L_M^+:=\{(x,y)\in\mathcal{R}:x=M\}.
\end{aligned}
\end{equation*}
We will first consider the two domains independently and then glue them back together for the proof of Proposition~\ref{prop: open}. We begin with the bounded rectangle. 
\begin{lemma}\label{lem: all nodal}
	Fix a solution $(\eta^*,\zeta^*,\alpha^*)$ to \eqref{fullproblem} with associated velocity field $(u^*,v^*)$ from \eqref{conformal velocity}, and suppose that $v^*$ satisfies \eqref{v<0} and $\alpha^*<1-\gamma$. For any $M>0$, there exists an $\ep_M>0$ such that for all solutions $(\eta,\zeta,\alpha)$ to \eqref{fullproblem} with  $\|\eta-\eta^*\|_{C^3(\mathcal{R})}+\|\zeta-\zeta^*\|_{C^3(\mathcal{R})}+|\alpha-\alpha^*|<\ep_M,$ the corresponding vertical velocity $v$  satisfies 
	\begin{equation}\label{open in rectangle}
	v<0\qquad\text{in }(\mathcal{R}\cup\Gamma)\cap\{0<x\leq2M\}.
	\end{equation}	
\end{lemma}
\begin{proof}
	Let $\eta^*,\zeta^*,\alpha^*, u^*$ and $v^*$ be as in the statement of the lemma. We begin by considering the following inequalities
	\begin{subequations}\label{v nodal}
		\begin{alignat}{2}
		\label{vx<0}
		v^*_x&<0&\qquad&\text{on }L^+\cup\{(0,1)\},\\
		\label{vy<0}
		v^*_y&<0&\qquad&\text{on }B^+\cap\{0< x\leq 2M\},\\
		\label{vxy<0}
		v^*_{xy}&<0&\qquad&\text{at }(0,0).
		\end{alignat}
	\end{subequations} 
  and claim that these are simply a consequence of \eqref{v<0} and \eqref{fullproblem}. Indeed, let us first prove that \eqref{vx<0} and \eqref{vy<0} hold. To this end, we differentiate \eqref{boundary v} with respect to $x$. Evaluating it at the wave crest $(0,1)$, we get
	\begin{equation}\label{wave crest}
	(u^*)^2v_{xy}^*-u^*(v_x^*)^2-\eta_y^*(\gamma u^*+\alpha^*)v_x^*=0\qquad\text{at }(0,1),
	\end{equation} 
	after some cancellations. Moreover, since $v^*$ is harmonic and odd in $x$, at $(0,1)$ we must have $v^*=v_{xx}^*=v_{yy}^*=v_y^*=0$. By \eqref{wave crest}, having $v_x^*=0$ at this point would additionally imply $v_{xy}^*=0$, contradicting Serrin's edge point lemma; see \cite[Lemma 1]{serrin}. Moreover, since $v^*=0$ along $L^+$, the Hopf boundary point lemma ensures that $v_x^*<0$ on $L^+$. Similarly, since $v^*$ vanishes along $B^+$, the Hopf lemma gives $v_y^*<0$ on $B^+\cap\{0< x\leq 2M\}$.
	
	It remains to prove \eqref{vxy<0} at the corner point $(0,0)$. Exactly as above we find that $v^*=v_{xx}^*=v_y^*=v_{yy}^*=0$ at this point. Moreover, \eqref{vy<0} implies that $v_{xy}^*\leq0$ there. However, since $v_x^*=0$ along $B$, $v_{xy}^*=0$ at $(0,0)$ would contradict Serrin's lemma and hence we must have $v_{xy}^*<0$ at $(0,0)$.
	
	By arguing as in \cite{cs:stag}, it can be shown that \eqref{open in rectangle} combined with \eqref{v nodal} defines an open set in $C^2(\mathcal{R}^+\cap\{0\leq x\leq 2M\})$. Thus, $v$ also satisfies \eqref{open in rectangle} and \eqref{v nodal} when $\ep_M$ is sufficiently small.
\end{proof}

We now turn to the semi-infinite strip $\mathcal{R}_M^+$. In particular, we have the following result which will allow us to show that \eqref{v<0}  holds in a neighborhood of infinity.

\begin{lemma}\label{lem: smallstrict}
	Fix $\alpha_0\in(0,1-\gamma)$. Then there exists a $\delta=\delta(\alpha_0)>0$ such that the following holds. Let $(\eta,\zeta,\alpha)$ solve \eqref{fullproblem} with $0<\alpha\leq\alpha_0$, and define $v$ by \eqref{conformal velocity} and $w$ by \eqref{tilde}. If, for some $M>0$, $\|w\|_{C^1(\mathcal{R}_M^+)}<\delta$ and $v\leq0$ on $L_M^+$, then $v<0$ in $\mathcal{R}_M^+\cup\Gamma_M^+$, or else $v\equiv0$.
\end{lemma}

\begin{proof}
	We choose $\ep,\delta>0$ sufficiently small so that
	\begin{equation}\label{size ep delta}
		b:=\frac{u^2-\eta_y(\gamma u+\alpha)(1+\ep)}{u^2}>0\qquad\text{and}\qquad u>0\text{ on }\Gamma_M^+.
	\end{equation} 
	This can be done since as $\ep,\delta\to0$ we have $u\to1$ and $b\to1-\gamma-\alpha$, uniformly on $\Gamma_M^+$. Consider now the auxiliary function 
	\begin{equation*}
	f:=\frac{v}{y+\ep}.
	\end{equation*}
	Clearly $f$ and $v$ must have the same sign, and $f$ vanishes at infinity by \eqref{asymptotic velocity}. An easy calculation using \eqref{v problem} shows that $f$ satisfies the elliptic problem
	\begin{subequations}\label{elliptic f}
	\begin{alignat}{2}
	\Delta f+\frac{2}{y+\ep}f_y&=0&\qquad&\text{in }\mathcal{R}_M^+\\
	\label{f boundary}
	(1+\ep)u^2f_y-(1+\ep)uvf_x+bf&=0&\qquad&\text{on }\Gamma_M^+,\\
	f&=0&\qquad&\text{on }B_M^+,
	\end{alignat}
	\end{subequations} 
  where, by \eqref{size ep delta}, the coefficients in front of $f$ and $f_y$
  in \eqref{f boundary} are strictly positive. Suppose for the sake of
  contradiction that $f \not \equiv 0$ achieves its nonnegative supremum over $\mathcal{R}_M^+$ at some point $(x_0,y_0) \in \mathcal{R}_M^+ \cup
  \Gamma_M^+$. By the strong maximum principle, we must have $(x_0,y_0) \in \Gamma_M^+$, and so the Hopf lemma implies $f_y > 0$ there. Since by assumption $f \ge 0$ at this point, the left hand side of \eqref{f boundary} is then strictly positive, which is a contradiction.
\end{proof}

We are now ready to glue the domains back together and prove Proposition~\ref{prop: open}. 

\begin{proof}[Proof of Proposition~\ref{prop: open}]
	Fix $\eta^*,\zeta^*$ and $\alpha^*$  as in the statement and recall that the corresponding $w^*$ is defined in \eqref{tilde}. Moreover, fix $\alpha_0\in(\alpha^*,1-\gamma)$ and choose $M>0$ such that $\|w^*\|_{C^1(\mathcal{R}_M^+)}<\tfrac{1}{2}\delta$, where $\delta=\delta(\alpha_0)>0$ is as in Lemma~\ref{lem: smallstrict}. Finally, pick $\ep_M>0$ such that Lemma~\ref{lem: all nodal} holds for $w=w^*$. Choosing $\ep:=\min(\ep_M,\tfrac{1}{2}\delta,|\alpha_0-\alpha^*|)$ ensures, by Lemma~\ref{lem: all nodal}, that \eqref{open in rectangle} holds. In particular, $v\leq0$ on $L_M^+$. Since $\|w\|_{C^1(\mathcal{R}_M^+)}<\delta$, Lemma~\ref{lem: smallstrict} yields $v<0$ in $\mathcal{R}_M^+\cup\Gamma_M^+$. Combining with \eqref{open in rectangle} gives the result.
\end{proof}

\section{Functional analytic formulation and linearized operators}\label{sect: Functional Analytic}
We now express \eqref{fullproblem} as a nonlinear operator equation in a suitable Banach space and analyze the associated linearized operators. For convenience, we will work with problem \eqref{fullproblem} expressed in terms of $w=(w_1,w_2)$, as defined in \eqref{tilde}. This will put us in the proper setting to prove the existence results in Section~\ref{sect: Existence Results}.  

\subsection{Functional analytic formulation}\label{subsect: function spaces}
We will work in the Banach spaces
\begin{equation*}
\begin{aligned}
\mathcal{X}&=\{w\in  C^{3+\beta}_{\be}(\overline{\mathcal{R}})\cap C^{2}_0(\overline{\mathcal{R}}):\Delta w=0\text{ in }\mathcal{R},\; w=0\text{ on }B  \},\\
\mathcal{Y}&=(C_{\mathrm{b}}^{3+\beta}(\Gamma)\cap C^2_0(\Gamma))\times (C_{\mathrm{b}}^{2+\beta}(\Gamma)\cap C^{1}_0(\Gamma)),
\end{aligned}
\end{equation*}
as well as in the larger spaces $\mathcal{X}_\be$ and $\mathcal{Y}_\be$ whose elements do not necessarily vanish at infinity, that is, the intersections with $C_0^1$ and $C_0^2$ are removed.
Moreover, we define the open subset
\begin{align}
  \label{eqn:U}
  \mathcal{U}:=\Big\{(w,\alpha)\in\mathcal{X}\times\mathbb{R}: \alpha<\alpha_\crit,\,\, \lambda(w,\alpha)>0 \Big\}\subset\mathcal{X}\times\mathbb{R},
\end{align}
where
\begin{alignat}{2}\label{eqn:lambda}
\lambda(w,\alpha):=\inf_{\mathcal{R}}4(1-2\alpha w_1)^2\big(w_{1x}^2+(1+w_{1y})^2\big)>0
\end{alignat}
is the assumption \eqref{no stagnation on free surface} expressed in terms of $w$. We will see that $\lambda(w,\alpha)$ is closely related to the \emph{minor}, or \emph{Lopatinskii}, constant for the linearized problem, as referred to in \cite{adn2}. In particular, \eqref{eqn:lambda} is a sufficient condition for being able to apply linear Schauder estimates to the linearized problem. 

We can now interpret \eqref{fullproblem} as the nonlinear operator equation
\begin{equation}\label{fullproblem operator}
\mathscr{F}(w,\alpha)=0,
\end{equation}
where $\mathscr{F}:=(\mathscr{F}_1,\mathscr{F}_2):\mathcal{X}\times\mathbb{R}\rightarrow \mathcal{Y}$ is given by
\begin{equation}\label{operator F}
\begin{aligned}
\mathscr{F}_1(w,\alpha)&=(w_2+\gamma w_1+\tfrac{1}{2}\gamma w_1^2),\\
\mathscr{F}_2(w,\alpha)&=\big(\gamma(w_1+w_{1y}+w_1w_{1y})+w_{2y}+1\big)^2-(1-2\alpha w_1)\big(w_{1x}^2+(w_{1y}+1)^2\big).
\end{aligned}
\end{equation}
Linearizing this operator, we get
\begin{equation}\label{linearized operator}
\begin{aligned}
\mathscr{F}_{1w}(w,\alpha)\dot{w}&=c_i\dot{w}_i,\\
\mathscr{F}_{2w}(w,\alpha)\dot{w}&=a_{ij}\partial_i\dot{w}_j+b_i\dot{w}_i,
\end{aligned}
\end{equation}
where $\partial_1=\partial_x$ and $\partial_2=\partial_y$, and where the coefficients are given by
\begin{equation}\label{coefficients}
	\begin{aligned}
	a_{11}&=-2(1-2\alpha w_1)w_{1x},\\
	a_{12}&=2\gamma\big(w_{2y}+1+\gamma(w_1+w_{1y}+w_1w_{1y})\big)(1+w_1)-2(1-2\alpha w_1)(1+w_{1y}),\\
	a_{21}&=0,\\
	a_{22}&=2\big(w_{2y}+1+\gamma(w_1+w_{1y}+w_1w_{1y})\big),\\
	b_1&=2\gamma(w_{1y}+1)\big(w_{2y}+1+\gamma(w_1+w_{1y}+w_1w_{1y})\big)+2\alpha(w_{1x}^2+(1+w_{1y})^2),\\
	b_2&=0,\\
	c_1&=\gamma(1+w_1),\\
	c_2&=1.
	\end{aligned}
\end{equation}
Note that \eqref{linearized operator} is the type of operator considered in Appendix~\ref{app: schauder estimates}. Linearizing about the trivial solution $w\equiv0$ we obtain
\begin{subequations}\label{linearized operator trivial}
\begin{align}
\mathscr{F}_{1w}(0,\alpha)\dot{w}&=\gamma\dot{w}_1+\dot{w}_2,\\
\mathscr{F}_{2w}(0,\alpha)\dot{w}&=2\big(\dot{w}_{2y}+(\gamma-1)\dot{w}_{1y}+(\gamma+\alpha)\dot{w}_1\big).
\end{align}
\end{subequations}
\begin{remark}\label{rem:limiting operator}
	Thanks to \eqref{tilde regularity}, the \emph{limiting operator} of $\mathscr{F}_w(w,\alpha)$, obtained by sending $x\to\pm\infty$ in the coefficients, is none other than $\mathscr{F}_w(0,\alpha)$.
\end{remark}

\subsection{Local properness and invertibility properties}\label{subsect: linearized op}
In this subsection, we closely study the linearized operator $\mathscr{F}_w(0,\alpha)\colon\mathcal{X}\to\mathcal{Y}$ given in \eqref{linearized operator trivial} when $\alpha<\alpha_\crit$ is supercritical. In particular, we will show that in this regime, $\mathscr{F}_w(0,\alpha)$ is injective $\mathcal{X}_\be\to\mathcal{Y}_\be$. Using standard translation arguments and Schauder estimates, we then get that $\mathscr{F}_w(0,\alpha)$ is \emph{locally proper}. That is, the pre-image under $\mathscr{F}_w(0,\alpha)$ of any compact set intersected with a closed and bounded set in the domain is compact. This is equivalent to $\mathscr{F}_w(0,\alpha)$ being semi-Fredholm with finite dimensional kernel and closed range. From this we show that $\mathscr{F}_w(0,\alpha)$ is Fredholm with index $0$ and that the full linearized operator $\mathscr{F}_w(w,\alpha)$ is locally proper $\mathcal{X}\to\mathcal{Y}$.

\begin{remark}
	A general reference on elliptic problems and Fredholmness in unbounded domains is \cite{volpert:book}. See also \cite{paper} and \cite{cww:globalfronts}.
\end{remark}

Before we begin, we must show that we can apply Schauder estimates to the linearized operator $\mathscr{F}_w(w,\alpha)$. 
\begin{lemma}\label{lem: schauder applied}
	For $(w,\alpha)\in\mathcal{U}$, the linearized operator $\mathscr{F}_w(w,\alpha)\colon\mathcal{X}_\be\to\mathcal{Y}_\be$ in \eqref{linearized operator} enjoys the Schauder estimate 
	\begin{equation*}
	\|\dot{w}\|_{\mathcal{X}_\be}\leq C\big(\|\mathscr{F}_{w}(w,\alpha)\dot{w}\|_{\mathcal{Y}_\be}+\|\dot{w}\|_{C^0(\mathcal{R})} \big),
	\end{equation*}
	where the constant $C$ depends on $\|w\|_\mathcal{X}$, on $\alpha$ and on the minor constant $\lambda(w,\alpha)$.
\end{lemma}
\begin{proof}
	The proof follows directly from Lemma~\ref{lem: schauder estimates} from the Appendix provided we can show the uniform bound 
	\begin{equation}\label{complementing condition bound reflected}
		\inf_\Gamma\big((c_1a_{21}-c_2a_{11})^2+(c_1a_{22}-c_2a_{12})^2\big)>0.
	\end{equation} 
	Using the coefficients in \eqref{coefficients}, we calculate the left-hand side of \eqref{complementing condition bound reflected} to be $\inf_\Gamma4(1-2\alpha w_1)^2\big(w_{1x}^2+(1+w_{1y})^2\big)$. This is bounded below by $\lambda(w,\alpha)$ as defined in \eqref{eqn:lambda}, thus concluding the proof.
\end{proof}
The next lemma yields the sense in which $\alpha_\crit$ as defined in \eqref{alpha crit} is the critical wave speed parameter. Specifically, we will see that at the value $\alpha=\alpha_\crit$ the linear operator $\mathscr{F}_w(0,\alpha)$ is singular in the sense that it has a one-dimensional kernel. Due to our unbounded domain, this result combined with a translation argument shows that $\mathscr{F}_w(w,\alpha_\crit)$ is a non-Fredholm map from $\mathcal{X}\to\mathcal{Y}$ (see \cite{volpert:book} for more details). By contrast, in the periodic setting, Schauder estimates would be sufficient for proving that the linearized operator is semi-Fredholm.
\begin{lemma}\label{lem: trivial kernel}
The linear operator $\mathscr{F}_w(0,\alpha)\colon\mathcal{X}_\be\to\mathcal{Y}_\be$ given in \eqref{linearized operator trivial} has a trivial kernel if and only if $\alpha<\alpha_\crit$. 	
\end{lemma} 
\begin{proof}
	Using separation of variables, it is enough to rule out solutions $(\dot{w}_1,\dot{w}_2)\in\mathcal{X}_\be$ of the form 
	\begin{equation}\label{kernel}
	\begin{aligned}
    \dot{w}_1&=c_1\cos(kx)\sinh(ky),\\
    \dot{w}_2&=c_2\cos(kx)\sinh(ky),
	\end{aligned}	
	\end{equation}
  for some real wave number $k$ and constants $c_1$ and $c_2$. Plugging \eqref{kernel} into the boundary conditions in \eqref{linearized operator trivial}, some algebra shows that we have non-trivial solutions if and only if the dispersion relation
	\begin{equation}\label{dispersion relation}
	\gamma+\alpha=k\coth(k)
	\end{equation}
	holds. Since $k\coth(k)$ attains its minimum value of $1$ at $k=0$, \eqref{dispersion relation} has no real solutions if $\alpha+\gamma<1$. 
\end{proof}
In other words, Lemma~\ref{lem: trivial kernel} implies that $\mathscr{F}_w(0,\alpha)$ is an injective mapping from $\mathcal{X}_\be$ to $\mathcal{Y}_\be$ for $\alpha<\alpha_\crit$. We now turn to the question of local properness and invertibility, made more subtle by the unboundedness of the domain.

\begin{lemma}\label{lem: locally proper full}
	For $(w,\alpha)\in\mathcal{U}$, the linearized operator $\mathscr{F}_w(w,\alpha)$ is locally proper both $\mathcal{X}_\be\to\mathcal{Y}_\be$ and $\mathcal{X}\to\mathcal{Y}.$
\end{lemma}
\begin{proof}
	By using a translation argument and Schauder estimates, one can show that $\mathscr{F}_w(w,\alpha)$ is locally proper $\mathcal{X}_\be\to\mathcal{Y}_\be$ if and only if the limiting operator $\mathscr{F}_w(0,\alpha)\colon\mathcal{X}_\be\to\mathcal{Y}_\be$ (see Remark~\ref{rem:limiting operator}) has a trivial kernel \cite{volpert:book}. Therefore, by Lemma~\ref{lem: trivial kernel} the linearized operator $\mathscr{F}_w(w,\alpha)$ is locally proper $\mathcal{X}_\be\to\mathcal{Y}_\be$.
	
	It remains to show that $\mathscr{F}_w(w,\alpha)$ is also locally proper $\mathcal{X}\to\mathcal{Y}$. Let $\dot{w}_n\in\mathcal{X}$ be a bounded sequence such that $\mathscr{F}_w(w,\alpha)\dot{w}_n$ is a convergent sequence in $\mathcal{Y}$. Since $\mathscr{F}_w(w,\alpha)\colon\mathcal{X}_\be\to\mathcal{Y}_\be$ is locally proper, we can extract a subsequence so that $\dot{w}_n\to \dot{w}\in\mathcal{X}_\be$. Since $\mathcal{X}$ is a closed subspace, we must have $\dot{w}\in\mathcal{X}$, thus concluding the proof. 
\end{proof}

Finally, we show that for the linearized operator at the trivial solution $\mathscr{F}_w(0,\alpha)$ we get a stronger result, namely invertibility.

\begin{lemma}\label{lem: invertible}
	For $(w,\alpha)\in\mathcal{U}$, the linear operator   $\mathscr{F}_w(0,\alpha)$ is invertible both $\mathcal{X}_\be\to\mathcal{Y}_\be$ and $\mathcal{X}\to\mathcal{Y}$.
\end{lemma}

\begin{proof}
	From Lemmas~\ref{lem: trivial kernel} and \ref{lem: locally proper full} we know that $\mathscr{F}_w(0,\alpha)$ has a trivial kernel and is locally proper both $\mathcal{X}_\be\to\mathcal{Y}_\be$ and $\mathcal{X}\to\mathcal{Y}$. In order to prove invertibility $\mathcal{X}_\be\to\mathcal{Y}_\be$ it therefore suffices to check that the operator is Fredholm index $0$. To this end, we introduce the linear operator $\mathscr{L}(t)\colon\mathcal{X}_\be\to\mathcal{Y}_\be$, defined by
	\begin{equation*}
		\mathscr{L}(t)\dot{w}:=
		\begin{pmatrix}
		\gamma\dot{w}_1+\dot{w}_2\\
		2\big(\dot{w}_{2y}+(\gamma-1)\dot{w}_{1y}+(\gamma+\alpha)t\dot{w}_1\big)
		\end{pmatrix},\qquad\text{for }t\in[0,1].
	\end{equation*}
	A variant of the argument in the proof of Lemma~\ref{lem: trivial kernel} shows that $\mathscr{L}(t)$ is injective from $\mathcal{X}_\be$ to $\mathcal{Y}_\be$, provided that $\alpha\in(0,\alpha_\crit)$. Indeed, the dispersion relation \eqref{dispersion relation} becomes
	\begin{equation*}
		t(\gamma+\alpha)=k\operatorname{coth}(k)
	\end{equation*}
	which implies that the kernel is trivial if we have $t(\gamma+\alpha)<1$. Arguing as in Lemma~\ref{lem: locally proper full} we therefore get that $\mathscr{L}(t)\colon\mathcal{X}_\be\to\mathcal{Y}_\be$ is locally proper and hence semi-Fredholm with trivial kernel for all $t\in[0,1]$. Specifically, this implies that the Fredholm index of $\mathscr{L}(t)$ is constant for all $t\in[0,1]$. However, we know that $\mathscr{L}(0)$ is invertible (follows for instance from \cite[Theorem A.8]{cww:globalfronts}) and therefore has Fredholm index $0$. By continuity of the index, this implies that $\mathscr{L}(1)=\mathscr{F}_w(0,\alpha)$ must also have Fredholm index $0$ and therefore, $\mathscr{F}_w(0,\alpha)\colon\mathcal{X}_\be\to\mathcal{Y}_\be$ is invertible. 
	
	Finally, we show that $\mathscr{F}_w(0,\alpha)$ is also invertible $\mathcal{X}\to\mathcal{Y}$. Since $\mathcal{X}\subset\mathcal{X}_\be$, clearly $\mathscr{F}_w(0,\alpha)\colon\mathcal{X}\to\mathcal{Y}$ is injective. In order to show surjectivity, let $\dot{f}\in\mathcal{Y}$. Since $\mathscr{F}_w(0,\alpha)\colon\mathcal{X}_\be\to\mathcal{Y}_\be$ is invertible, there exists a $\dot{w}\in\mathcal{X}_\be$ with $\mathscr{F}_w(0,\alpha)\dot{w}=\dot{f}$. Using a translation argument as in \cite[Lemma A.10]{pressure}, we get that $\dot{w}\in\mathcal{X}$.  
\end{proof}

\section{Uniform regularity}\label{sect: Uniform Regularity}
This section is devoted to the proof of the following proposition, which states that the $C^{3+\beta}$ norms of $\eta$ and $\zeta$ can be controlled in terms of the constant $\delta>0$ in the inequalities
\begin{equation}\label{eqn:delta}
\delta\leq|\nabla\eta|\leq1/\delta\quad\text{and}\quad1-2\alpha(\eta-1)\geq\delta\qquad\text{in }\mathcal{R}.
\end{equation}
Note that the left hand side of \eqref{no stagnation on free surface}, and hence the Lopatinskii constant $\lambda$ in \eqref{eqn:lambda}, can be controlled solely in terms of $\delta$.
\begin{proposition}\label{prop: uniform regularity}
	Suppose that $(\eta,\zeta,\alpha)$ solves \eqref{fullproblem} with $0\leq\alpha\leq\alpha_\crit$ and that \eqref{eqn:delta} holds for some $\delta>0$. Then there exists a positive constant $C=C(\delta)$ such that $\|\eta\|_{C^{3+\beta}(\mathcal{R})}<C$ and $\|\zeta\|_{C^{3+\beta}(\mathcal{R})}<C$.
\end{proposition}

In order to prove this proposition, we split \eqref{fullproblem} into the following two coupled scalar problems to which we will apply the classical results from \cite{lieberman}. It will be more convenient to do so whilst working with $\psi$, the stream function in conformal variables, rather than with $\zeta$. Recall that
\begin{equation}\label{definition w}
  \psi(x,y)=:\zeta(x,y)+\tfrac{\gamma}{2}\eta^2(x,y).
\end{equation}
Specifically, we consider \eqref{harmonic zeta}--\eqref{kinematic} together with \eqref{bottom zeta} as a problem for $\psi$ with fixed $\eta$,
\begin{subequations}\label{w problem}
	\begin{alignat}{2}
	\label{w interior}
	\Delta \psi&=\gamma|\nabla\eta|^2&\qquad&\text{in }\mathcal{R},\\
	\label{w surface}
	\psi&=1-\tfrac{1}{2}\gamma&\qquad&\text{on }\Gamma,	\\
	\psi&=0&\qquad&\text{on }B,
	\end{alignat}
\end{subequations}
and \eqref{harmonic eta} combined with \eqref{dynamic}--\eqref{bottom eta} as a problem for $\eta$ with fixed $\psi$,
\begin{subequations}\label{eta problem}
	\begin{alignat}{2}
	\Delta\eta&=0&\qquad&\text{in }\mathcal{R},\\
	\label{eta surface}
	(1-2\alpha(\eta-1))|\nabla\eta|^2&=\psi_y^2&\qquad&\text{on }\Gamma,\\
	\eta&=0&\qquad&\text{on }B.
	\end{alignat}
\end{subequations}

We use these two problems to perform a bootstrap argument for which the key ingredient is applying \cite[Theorems 1 and 3]{lieberman} to the scalar problem \eqref{eta problem}. To this end, we reformulate the boundary condition \eqref{eta surface} as 
\begin{equation*}
 G(x,\eta,D\eta;\alpha)=0\quad\text{on }\Gamma,
\end{equation*}
where
\begin{equation*}
 G(x,z,p;\alpha):=|p|^2-\frac{\psi_y^2(x,0)}{1-2\alpha(z-1)}.
\end{equation*} 
Notice that when restricted to the set
\begin{equation*}\label{O delta}
\mathcal{O}_\delta=\big\{((x,y),z,p)\in\mathcal{R}\times\mathbb{R}\times\mathbb{R}^2:\delta<|p|<1/\delta,\;1-2\alpha(z-1)>\delta,\;y\in(\tfrac 12,1) \big\},
\end{equation*} 
$G$ is smooth and satisfies the inequality
\begin{equation}\label{Gp}
|G_p(x,z,p;\alpha)|>2\delta.
\end{equation}
The proof of Proposition~\ref{prop: uniform regularity} consists of the combination of the following three lemmas.

\begin{lemma}\label{lem: lemma 1}
	For any $\delta>0$,  there exist constants $C>0$ and $\sigma>0$ such that a solution $(\eta,\zeta,\alpha)$ to \eqref{fullproblem} satisfying \eqref{eqn:delta} and $0\leq\alpha\leq\alpha_\crit$ also satisfies $\|\eta\|_{C^{1+\sigma}(\mathcal{R})}<C$. 
\end{lemma} 
\begin{proof}
	Applying \cite[Theorem 8.33]{GT} to \eqref{w problem} yields, for instance $\|\psi\|_{C^{1+1/2}(\mathcal{R})}<C$. Applying \cite[Theorem 1]{lieberman} to \eqref{eta problem} (see comment after \cite[Theorem 2]{lieberman}), along with \cite[Theorem 8.29]{GT} for the bottom boundary, concludes the proof. 
\end{proof}
\begin{lemma}\label{lem: uniform reg 2}
	For any $\delta>0$, $\ep>0$ and $K>0$ there exist constants $C>0$ and $\sigma>0$  such that any solution $(\eta,\zeta,\alpha)$ to \eqref{fullproblem} satisfying $0\leq\alpha\leq\alpha_\crit$, \eqref{eqn:delta} and $\|\eta\|_{C^{1+\ep}(\mathcal{R})}<K$, also satisfies $\|\eta\|_{C^{2+\sigma}(\mathcal{R})}<C$. 
\end{lemma}
\begin{proof}
	In what follows, $C$ and $\sigma$ will denote positive constants which depend on $K$, $\delta$ and $\ep$ but can vary from line to line.	Let us now consider \eqref{w problem}. Notice that the $C^{\ep}$ norm of the right hand side of \eqref{w interior} is controlled by $C$. Applying basic Schauder estimates then yields $\|\psi\|_{C^{2+\ep}(\mathcal{R})}<C$. Returning to \eqref{eta problem}, we therefore obtain that the $C^{1+\ep}$ norm of the right hand side of \eqref{eta surface} is controlled by $C$. In order to show that this implies $\|\eta\|_{C^{2+\sigma}(\mathcal{R})}<C$, we wish to apply \cite[Theorem 3]{lieberman}. While this theorem is stated for problems where \eqref{Gp} holds globally, this restriction can be overcome, for instance by constructing the extension $\overline G$ of $G$ in \cite[Lemma 2]{lieberman} using mollifications in $x$ alone. Hence, \cite[Theorem 3]{lieberman}, combined with \cite[Theorem 8.29]{GT} for the bottom boundary, yields $\|\eta\|_{C^{2+\sigma}(\mathcal{R})}<C$. 
\end{proof} 
The last step is more routine and relies only on linear Schauder estimates. 
\begin{lemma}\label{lem: lemma 3}
	For any $\delta>0$, $\ep\in(0,\beta]$ and $K>0$ there exists a constant $C>0$ depending only on $K,\ep$ and $\delta$ such that a solution $(\eta,\zeta,\alpha)$ to \eqref{fullproblem}, satisfying $0\leq\alpha\leq\alpha_\crit$, \eqref{eqn:delta} and $\|\eta\|_{C^{2+\ep}(\mathcal{R})}<K$ also satisfies $\|\eta\|_{C^{3+\beta}(\mathcal{R})}<C$ and $\|\zeta\|_{C^{3+\beta}(\mathcal{R})}<C$. 
\end{lemma}
\begin{proof}
	Once again, we let $C$ denote a positive constant which depends on $K$, $\ep$ and $\delta$ but can vary from line to line. Since $\|\eta\|_{C^{2+\ep}(\mathcal{R})}<K$ by assumption, we get $\|\psi\|_{C^{3+\ep}(\mathcal{R})}<C$ by considering the elliptic problem \eqref{w problem}. From \eqref{definition w} we therefore immediately get $\|\zeta\|_{C^{2+\ep}(\mathcal{R})}<C$. In particular, $\|w\|_{C^{2+\ep}(\mathcal{R})}<C$ for $w=(w_1,w_2)$ defined as in \eqref{tilde}. 
	
	We now differentiate $\mathscr{F}(w,\alpha)=0$ (see \eqref{operator F}) with respect to $x$ and see that $\phi:=w_x$ is a pair of $C^{2+\ep}$ harmonic functions solving $\mathscr{F}_w(w,\alpha)\phi=0$ and $\phi=0$ on $B$. As in Lemma~\ref{lem: schauder applied}, the Schauder estimates from Lemma~\ref{lem: schauder estimates} in Appendix~\ref{app: schauder estimates} now yield $\|w_{x}\|_{C^{2+\ep}(\mathcal{R})}<C$, where here we have used that the Lopatinskii constant $\lambda$ in \eqref{eqn:lambda} is controlled by $\delta$.
	
	Finally, solving $\Delta w=0$ for $\partial_{yy}w$ we get $\|w\|_{C^{3+\ep}(\mathcal{R})}<C$. In particular, $\|w\|_{C^{2+\beta}(\mathcal{R})}<C$ . We can therefore repeat the above arguments with $\ep=\beta$. The definitions of $w_1$ and $w_2$ in \eqref{tilde} then yield $\|\eta\|_{C^{3+\beta}(\mathcal{R})}<C$ and $\|\zeta\|_{C^{3+\beta}(\mathcal{R})}<C$, as desired. \end{proof}

Proposition~\ref{prop: uniform regularity} follows easily from combining Lemmas~\ref{lem: lemma 1}, \ref{lem: uniform reg 2} and \ref{lem: lemma 3}. 

Finally, we conclude this section by proving the following result which will be useful in Section~\ref{sect: Existence Results} to winnow out alternatives. 

\begin{proposition}\label{prop: gamma cases}
	Assume that $\psi$ as defined in \eqref{definition w} solves \eqref{w problem} and that $(\eta,\zeta,\alpha)$ solve \eqref{fullproblem} with $0\leq\alpha\leq\alpha_\crit$. Then we have the following bounds:
	\begin{enumerate}[label=\rm(\roman*)]
		\item \label{gamma negative} if $\gamma\le 0$ then $\psi_y<1-\tfrac{1}{2}\gamma$ on $\Gamma$,
		\item \label{gamma positive} if $\gamma\geq0$ then $\psi_y>\min\big\{2-\gamma,\gamma\inf_{\mathcal{R}}|\nabla\eta|^2 \big\}$ on $\Gamma$.
	\end{enumerate}
\end{proposition}
\begin{proof}	 
	We will treat the two cases separately. Let us begin by assuming $\gamma\leq0$ and defining
	\begin{equation*}
	\tilde{\psi}=\psi-\big(1-\tfrac{1}{2}\gamma\big)y.
	\end{equation*}
	It is easy to see that $\tilde{\psi}$ solves
	\begin{subequations}
		\begin{alignat*}{2}
		\Delta\tilde{\psi}&=\gamma|\nabla\eta|^2&\qquad&\text{in }\mathcal{R},\\
		\tilde{\psi}&=0&\qquad&\text{on }B,\\
		\tilde{\psi}&=0&\qquad&\text{on }\Gamma.
		\end{alignat*}
	\end{subequations}
	Since $\gamma\leq0$, by the strong minimum principle $\psi$ must attain its minimum on the boundary. In particular, applying the Hopf boundary point lemma at any point on the boundary, we get
	\begin{equation*}\label{w_y bound}
	\tilde{\psi_y}=\psi_y-1+\tfrac{1}{2}\gamma<0,
	\end{equation*}
	which is exactly \ref{gamma negative}.
	
	We now turn to the case $\gamma\geq0$. Let us denote
	\begin{equation*}
	\bar{\psi}=\psi-My^2,
	\end{equation*} 
	where
	\begin{equation*}
	M=\min\big\{1-\tfrac{1}{2}\gamma,\tfrac{1}{2}\gamma\inf_{\mathcal{R}}|\nabla\eta|^2 \big\}.
	\end{equation*}
	Notice that since the assumption $0\leq\alpha\leq\alpha_\crit$ implies that $0<1-\gamma$, the constant $M$ is always positive. Moreover, $\bar{\psi}$ solves
	\begin{subequations}
		\begin{alignat*}{2}
		\Delta\bar{\psi}&=\gamma|\nabla\eta|^2-M&\qquad&\text{in }\mathcal{R},\\
		\bar{\psi}&=0&\qquad&\text{on }B,\\
		\bar{\psi}&=1-\tfrac{1}{2}\gamma-M&\qquad&\text{on }\Gamma.
		\end{alignat*}
	\end{subequations}
	Using the assumption $\gamma>0$ and the definition of $M$, the strong maximum principle yields that $\bar{\psi}$ is maximized at any point on the boundary. At any such point, the Hopf boundary point lemma tells us that
	\begin{equation*}
	\bar{\psi}_y=\psi_y-2M>0,
	\end{equation*} 
	proving \ref{gamma positive}.
\end{proof}

\section{Existence results}\label{sect: Existence Results}
In this section, we prove the existence of large-amplitude solitary wave solutions to \eqref{fullproblem}. We begin by constructing a curve of  small-amplitude solutions $\mathscr{C}_\text{loc}$ using center manifold theory. We then extend this local curve to a global curve $\mathscr{C}$ using an argument based on real-analytic global bifurcation theory. 
\subsection{Small-amplitude theory}\label{subsect: small-amplitude}
This subsection is devoted to the construction of small-amplitude waves. For solitary waves the linearized operator at the bifurcation point, $\mathscr{F}_w(0,\alpha_\crit)$, is not Fredholm. Therefore, the Crandall--Rabinowitz local bifurcation theorem used in \cite{cs:exact} and \cite{walsh:stratified} for periodic waves cannot be applied. Instead, we use a center manifold approach. Customarily, the idea is to treat the spatial $x$ variable as a variable of time, thus reformulating the problem into an evolution equation. One then constructs a two-dimensional center manifold controlled by a two-dimensional reduced equation. The homoclinic orbits of this equation are our desired small-amplitude solutions. This is the strategy employed in \cite{gw, paper, strat}. The major disadvantage of this approach is that, due to the nonlinear boundary conditions, many tedious changes of variables need to be performed to obtain the reduced equation. This would be particularly unpleasant given our elliptic system setting. Recently however, a new center manifold reduction theorem was derived in \cite{chen2019center}, which we will use instead. On the one hand, this provides  us with a comparatively simple method for obtaining the reduced ODE. On the other, it also allows us to choose a suitable projection (or, linear relationship between our original $w$ and the variable governed by the reduced equation) so that this ODE relates transparently to the original problem. 

To begin with, since we want a reduced ODE with a Froude number very close to the critical Froude number,  we set
\begin{equation*}
\alpha=\alpha^\ep:=\alpha_\crit-\ep
\end{equation*}
where $\ep>0$ is small and $\alpha_\crit$ is defined in \eqref{alpha crit}.
The main result of this subsection is the following.
\begin{theorem}\label{thm: small-amplitude}
	There exists a continuous one-parameter curve  
	\begin{equation*}
	\mathscr{C}_\textup{loc}=\{(w^\ep,\alpha^\ep):0<\ep\ll1 \}\subset\mathcal{X}\times\mathbb{R}
	\end{equation*} 
	of nontrivial symmetric solutions to $\mathscr{F}(w,\alpha^\ep)=0$ with the asymptotic expansion
	\begin{equation*}
	w_1^{\ep}(x,0)=\frac{3\ep}{\gamma^2-3\gamma+3}\operatorname{sech}^2\bigg(\frac{\sqrt{3\ep}}{2} x\bigg)+O(\ep^{2+\tfrac{1}{2}})
	\end{equation*}
	in $C_\be^{3+\beta}$. Moreover, the following properties hold: 
	\begin{enumerate}[label=\rm(\roman*)]
	\item \label{monotonicity small-amplitude} \textup{(Monotonicity)} the solutions on $\mathscr{C}_\textup{loc}$ satisfy the nodal property \eqref{nodal eta};
	
	\item \label{uniqueness small-amplitude} \textup{(Uniqueness)} if $w\in\mathcal{X}$ and $\ep>0$ are sufficiently small and if $w$ satisfies \eqref{nodal eta}, then $\mathscr{F}(w,\alpha^\ep)=0$ implies $w=w^\ep$;
	
	\item \label{invertibility small-amplitude} \textup{(Invertibility)} for all $0<\ep\ll1$, the linearized operator $\mathscr{F}_w(w^\ep,\alpha^\ep)$ is invertible. 
	\end{enumerate} 
\end{theorem}

We will prove Theorem~\ref{thm: small-amplitude} using the center manifold reduction results in \cite{chen2019center}. To state these results, we need weighted versions of the spaces $\mathcal X,\mathcal Y$ defined in Section~\ref{subsect: function spaces} allowing for exponential growth, namely 
\begin{align*}
  \mathcal{X}_\mu &:=\big\{ (w_1,w_2)\in  (C^{3+\beta}_{\mu}(\overline{\mathcal{R}}))^2:\Delta w_i=0\textup{ in }\mathcal{R},\, w_i=0\textup{ on }B  \big\},\\
  \mathcal{Y}_\mu &:=C_\mu^{3+\beta}(\Gamma)\times C_\mu^{2+\beta}(\Gamma),
\end{align*} 
where here the weighted H\"older spaces $C^{k+\beta}_\mu$ were defined in Section~\ref{subsect: notation}. For any $\mu > 0$, the linearized operator
\begin{align*}
  \mathscr{L}:=\mathscr{F}_w(0,\alpha_\crit),
\end{align*}
given explicitly in \eqref{linearized operator trivial}, extends to a bounded operator $\mathcal X_\mu \to \mathcal Y_\mu$. For $\mu > 0$ sufficiently small, by separating variables we find that the kernel of this operator is two-dimensional, given by
\begin{equation*}
  \ker\mathscr{L}=\left\{
    \begin{pmatrix}
      (A+Bx)\varphi_1(y) \\ (A+Bx)\varphi_2(y)
    \end{pmatrix}
    :A,B\in\mathbb{R} \right\},
\end{equation*}
where 
\begin{alignat*}{2}
  \varphi
  &= 
  \begin{pmatrix}
    \varphi_1(y) 
    \\
    \varphi_2(y) 
  \end{pmatrix}
  =
  \begin{pmatrix}
  y\\-\gamma y
  \end{pmatrix}.
\end{alignat*}
\begin{theorem}\label{thm: center manifold}
	Fix a constant $\mu>0$ sufficiently small.
  Then there exist neighborhoods $\textbf{U}\in\mathcal{X}\times\mathbb{R}$ and $\textbf{V}\subset\mathbb{R}^3$ of the origin and a coordinate map $\Upsilon=(\Upsilon^1(A,B,\alpha),\Upsilon^2(A,B,\alpha))$ satisfying
	\begin{equation*}
		\Upsilon\in C^3(\mathbb{R}^3,\mathcal{X}_\mu),\qquad \Upsilon(0,0,\alpha)=\Upsilon_A(0,0,\alpha)=\Upsilon_B(0,0,\alpha)=0\textup{ for all }\alpha
	\end{equation*}
	such that the following hold.
	\begin{enumerate}[label=\rm(\alph*)]
	\item Suppose that $(w,\alpha)\in \textbf{U}$ solves \eqref{fullproblem operator}. Then $q(x):=w_1(x,1)$ solves the second-order ODE
	\begin{equation}\label{ODE v''}
		q''=f(q,q',\alpha)
	\end{equation}
	where $f\colon\mathbb{R}^3\to\mathbb{R}$ is the $C^3$ mapping
	\begin{align}
    \label{f definition}
		f(A,B,\alpha):=\frac{d^2}{dx^2}\bigg|_{x=0}\Upsilon(A,B,\alpha)(x,1)
  \end{align}
  and has the Taylor expansion
  \begin{align}
    \label{f expansion}
    f(A,B,\alpha)
    = 3\ep A-\tfrac{3}{2}(\gamma^2-3\gamma+3)A^2
    + O((|A|+|B|)(|A|+|B|+|\ep|)^3).
  \end{align}
\item \label{thm: center manifold:recovery} Conversely, if $q\colon\mathbb{R}\to\mathbb{R}$ satisfies the ODE \eqref{ODE v''} and $(q(x),q'(x),\alpha)\in \textbf{V}$ for all $x$, then $q=w_1(\cdot,1)$ for a solution $(w,\alpha)\in \textbf{U}$ of the problem \eqref{fullproblem operator}. Moreover,
	\begin{equation*}
		w_i(x+\tau,y)=q(x)\varphi_i(y)+q'(x)\tau\varphi_i(y)+\Upsilon^i(q(x),q'(x),\alpha)(\tau,y),
	\end{equation*}
	for all $\tau\in\mathbb{R}$.  
	\end{enumerate}
\end{theorem}
\begin{proof}
  We will apply Theorem~1.1 in \cite{chen2019center}. This result mostly concerns scalar problems and does not immediately apply to systems of the type \eqref{fullproblem}. However, when expressed in terms of the new dependent variables
 \begin{equation*}
  \left\{
 	\begin{aligned}
 		\hat{w}_1(x,y)&=w_1(x,y),\\
 		\hat{w}_2(x,y)&=-2w_2(x,-y)+(1-2\gamma)w_1(x,-y),
 	\end{aligned}
  \right.
 \end{equation*} 
 the linearized version of \eqref{fullproblem} is of the required transmission type for which the theorem has been extended to in \cite[Section~2.7]{chen2019center}. 

  It remains only to calculate the expansion \eqref{f expansion}. To this end, we expand $\Upsilon$ as
  \begin{equation*}
    \Upsilon(A,B,\ep)=\Upsilon_{200}A^2+\Upsilon_{101}\ep A+O((|A|+|B|)(|A|+|B|+|\ep|)^3),
  \end{equation*}
  in $\mathcal{X}_\mu$, where the functions $\Upsilon_{ijk}$ satisfy the normalization $\Upsilon_{ijk}(0,0)=\partial_x\Upsilon_{ijk}(0,0)=0$. By \cite[Theorem~1.3]{chen2019center}, we can calculate the $\Upsilon_{ijk}$ by inserting the corresponding Taylor expansion for $w$,
  \begin{equation*}
    \begin{aligned}
    w_i&=(A+Bx)\varphi_i+A^2\Upsilon_{200}^{i}+\ep A\Upsilon_{101}^{i}+O((|A|+|B|)(|A|+|B|+|\ep|)^3),
    \end{aligned}
  \end{equation*}
  into \eqref{fullproblem operator}. The first boundary condition $\mathscr F_1(w,\alpha)=0$ becomes
  \begin{align*}
    A^2\big(\gamma\Upsilon_{200}^1+\Upsilon_{200}^2+\tfrac{1}{2}\gamma\big)+\ep A(\gamma\Upsilon_{101}^1+\Upsilon_{101}^2) 
    + O((|A|+|B|)(|A|+|B|+|\ep|)^3)
    &= 0,
  \end{align*}
  on $\Gamma$, while the second boundary condition $\mathscr F_1(w,\alpha)=0$ becomes
  \begin{align*}
    0 &= A^2\big((2\gamma-2)\partial_y\Upsilon_{200}^1+2\partial_y\Upsilon_{200}^2+2\Upsilon_{200}^1+3-2\gamma+\gamma^2\big)\\
    &\qquad
    +2\ep A\big((\gamma-1)\partial_y\Upsilon_{101}^1+\partial_y\Upsilon_{101}^2+\Upsilon_{101}^1-1\big) \\
    &\qquad
    + O((|A|+|B|)(|A|+|B|+|\ep|)^3)
  \end{align*}
  on $\Gamma$. Grouping like terms yields the two linear equations
  \begin{equation*}
  \mathscr{L}\Upsilon_{200}=
    \begin{pmatrix}
      -\tfrac{1}{2}\gamma\\-\gamma^2+2\gamma-3
    \end{pmatrix}
  \qquad\text{and}\qquad
  \mathscr{L}\Upsilon_{101}=
    \begin{pmatrix}
      0\\2
    \end{pmatrix}.
  \end{equation*}

  A direct calculation shows that, for any $s_1,s_2 \in \mathbb R$, the problem
  \begin{align*}
    \mathscr{L}\tilde{\Upsilon}=
    \begin{pmatrix}
      s_1 \\ s_2 
    \end{pmatrix},
    \qquad\text{with}\qquad 
    \tilde{\Upsilon}^1(0,0)
    = 
    \partial_x \tilde{\Upsilon}^1(0,0)
    = 0,
  \end{align*}
  is solved by
  \begin{align*}
    \tilde{\Upsilon}^1&=\tfrac{3}{4}(s_2-2s_1)x^2y-\tfrac{1}{4}(s_2-2s_1)y(y^2-1),\\
    \tilde{\Upsilon}^2&=-\gamma\tilde \Upsilon^1+s_y.
  \end{align*} 
  By the arguments in \cite{chen2019center}, this solution is unique. By choosing $s_1,s_2$ appropriately we therefore obtain explicit expressions for $\Upsilon_{200},\Upsilon_{101}$, and calculate that
  \begin{equation*}
  \partial_x^2\Upsilon_{200}^1(0,0)=-\tfrac{3}{2}(\gamma^2-3\gamma+3)\qquad\text{and}\qquad \partial_x^2\Upsilon_{101}^1(0,0)=3.
  \end{equation*}
  Inserting into the definition \eqref{f definition} of $f(A,B,\ep)$ yields the desired expansion \eqref{f expansion}.
\end{proof}

We are now ready to prove Theorem~\ref{thm: small-amplitude}.  

\begin{proof}[Proof of Theorem~\ref{thm: small-amplitude}]
  By Theorem~\ref{thm: center manifold}, it suffices to work with the reduced ODE \eqref{ODE v''}. Introducing the scaled variables
	\begin{equation*}
		x=|\ep|^{-\tfrac1 2}X,\quad q(x)=\ep Q(X),\quad q_x(x)=\ep|\ep|^{\tfrac12} Q(X),
	\end{equation*}
  the expansion \eqref{f expansion} yields
	\begin{equation}\label{scaled reduced ODE}
		Q_{XX}=P_{X}=3Q-\tfrac{3}{2}(\gamma^2-3\gamma+3)Q^2+O(|\ep|^{\tfrac12}(|Q|+|P|)).
	\end{equation}
	When $\ep=0$, we have have an explicit homoclinic orbit 
	\begin{equation*}
	Q(X)=\frac{3}{\gamma^2-3\gamma+3}\operatorname{sech}^2\bigg(\frac{\sqrt{3}}{2}X \bigg),
	\end{equation*}
	joining the point $(0,0)$ to itself, and which intersects the $Q$-axis at the point $(Q_0,0)$, see Figure~\ref{fig:phase}. Here
	\begin{equation*}
		Q_0=\frac{3}{\gamma^2-3\gamma+3}.
	\end{equation*}
 	In particular, the unstable and stable manifolds meet at $(Q_0,0)$. By the stable manifold theorem, the unstable manifold depends uniformly smoothly on $\ep$. Combining this with the fact that the reversibility symmetry of \eqref{scaled reduced ODE} is independent of $\ep$ yields that for sufficiently small $\ep$ the unstable manifold intersects the $Q$-axis transversally at a point close to $(Q_0,0)$ (see for example, Proposition 5.1 in \cite{kirch:res}). We therefore conclude that, for $0<\ep\ll1$, there exists a homoclinic orbit to the origin which is symmetric in the sense that $Q$ is even in $X$. This proves the existence of a one-parameter family of solutions $(w^\ep,\alpha^\ep)$.
 	
 	\begin{figure}
 		\centering
 		\includegraphics[scale=1.1]{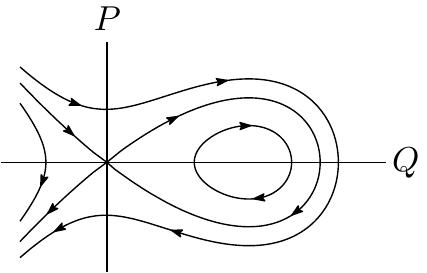}
 		\caption{Phase portrait of the scaled reduced ODE \eqref{scaled reduced ODE} when $\ep=0$
 			\label{fig:phase}
 		}
 	\end{figure}
 	
 	The same argument shows that the trajectory connecting $(0,0)$ to $(Q_0,0)$ also remains in the quadrant $\{Q>0,P>0\}$. We therefore get that $(w^\ep,\alpha^\ep)$ satisfies the monotonicity property 
	\begin{equation*}\label{w1x<0}
	w_{1x}^{\ep}>0\text{ on } \Gamma\cap\{x<0\}.
	\end{equation*}  
	From the strong maximum principle, we immediately get 
	\begin{equation*}
	w_{1x}^{\ep}>0\text{ in } (\Gamma\cup\mathcal{R})\cap\{x<0\}.
	\end{equation*} 
	Arguing similarly for $x>0$ we get
	\begin{equation*}
	w_{1x}^{\ep}<0\text{ in } (\Gamma\cup\mathcal{R})\cap\{x>0\}.
	\end{equation*} 
	Since $w_{1x}=\eta_x$ in $\overline{\mathcal{R}}$, we get that the nodal property \eqref{nodal eta} holds.  This concludes the proof of \ref{monotonicity small-amplitude}.
	
	We now show the uniqueness property \ref{uniqueness small-amplitude}. Let us assume that we have a solution $(w,\alpha^\ep)$ to the problem $\mathscr{F}_w(w,\alpha^\ep)=0$. By the properties of the center manifold, $w$ is determined by a homoclinic orbit of the reduced ODE. In particular, for this solution to be a wave of elevation satisfying the nodal property \eqref{nodal eta}, the homoclinic orbit needs to lie in the right half-plane $\{Q>0\}$. We must therefore necessarily have $w=w^\ep$, up to translation. Indeed, any solution $w$ which is not a translation of $w^\ep$ must lie in the left half-plane $\{Q<0\}$ for large $|x|$, and can therefore not be a wave of elevation. 
	
	In order to show the invertibility condition \ref{invertibility small-amplitude}, we use \cite[Theorem 1.6]{chen2019center} which tells us that  $\dot{w}\in\operatorname{ker}\mathscr{F}_w(w,\alpha^\ep)$ only if $\dot{q}=\dot{w}(\cdot,0)$ solves the linearized reduced ODE
	\begin{equation*}
		\dot{q}''=\nabla_{(q,q')} f(q,q',\ep)\cdot(\dot{q},\dot{q}').
	\end{equation*}
	The corresponding rescaled quantities $(\dot{Q},\dot{P})$ therefore solve the planar system
	\begin{equation}\label{dynamical system}
		\begin{pmatrix}
		\dot{Q} \\ \dot{P}
		\end{pmatrix}_X
		=\mathcal{M}(X)
		\begin{pmatrix}
		\dot{Q} \\ \dot{P}
		\end{pmatrix}
	\end{equation}
	with
	\begin{equation*}
		\lim_{X\to\pm\infty}\mathcal{M}(X)=
		\begin{pmatrix}
		0 & 1 \\ 3+O\big(\ep^{\tfrac{1}{2}} \big) & O\big(\ep^{\tfrac{1}{2}} \big)
		\end{pmatrix}.
	\end{equation*}
	As $X\to\pm\infty$, the eigenvalues of $\mathcal{M}$ approach $\pm\sqrt{3}$. Since this means that the eigenvalues are not purely imaginary, the dynamical system \eqref{dynamical system} only admits two linearly independent solutions: $q_1=w_x^\ep$, which we can rule out since it is not even, and $q_2$ which grows exponentially and is thus also not an admissible solution. Therefore, the only uniformly bounded solution to the reduced linearized problem is the trivial solution. Theorem 1.6 in \cite{chen2019center} therefore enables us to conclude that the kernel of $\mathscr{F}_w(w^\ep,\alpha^\ep)$ must be trivial. Moreover, by Lemma~\ref{lem: fredholm} below, the linearized operator $\mathscr{F}_w(w^\ep,\alpha^\ep)$ is Fredholm of index $0$ $\mathcal{X}\to\mathcal{Y}$, and must therefore be invertible. This concludes the proof of \ref{invertibility small-amplitude}.
\end{proof}
\begin{lemma}\label{lem: fredholm}
	Let $\delta_1,\delta_2>0$ be sufficiently small. Then for any solution $(w,\alpha)\in\mathcal{U}$ to \eqref{fullproblem operator} with $\|w\|_\mathcal{X}<\delta_1$ and $0<\alpha_\crit-\alpha<\delta_2$, the linearized operator $\mathscr{F}_w(w,\alpha)$ is Fredholm with index $0$ both $\mathcal{X}_\be\to\mathcal{Y}_\be$ and $\mathcal{X}\to\mathcal{Y}.$
\end{lemma}
\begin{proof}
	Let $w$ and $\alpha$ be as in the statement of the lemma. Note that $\delta_2>0$ implies $\alpha<\alpha_\crit$. We can therefore use the results from Section~\ref{sect: Functional Analytic}. Indeed,
  the limiting operator of $\mathscr{F}_w(w,\alpha)$ is $\mathscr{F}_w(0,\alpha)$ which is invertible by Lemma~\ref{lem: invertible}. Moreover, from Lemma~\ref{lem: locally proper full}, we know that $\mathscr{F}_w(w,\alpha)$ is semi-Fredholm $\mathcal{X}_\be\to\mathcal{Y}_\be$ and $\mathcal{X}\to\mathcal{Y}$. It remains to show that the Fredholm index of $\mathscr{F}_w(w,\alpha)$ is equal to $0$. By choosing $\delta_1$ and $\delta_2$ sufficiently small, we can ensure that the linearized operator $\mathscr{F}_w(w,\alpha)$ is uniformly elliptic and the corresponding minor constant $\lambda(w,\alpha)$ is uniformly bounded away from $0$. Similarly, we can guarantee a lower bound on the Lopatinskii constant \eqref{complementing condition bound} in Appendix~\ref{app: schauder estimates} for the family of operators 
	\begin{equation*}
	\mathscr{L}^{(t)}=\mathscr{F}_w(0,\alpha)+t(\mathscr{F}_w(w,\alpha)-\mathscr{F}_w(0,\alpha))
	\end{equation*}
	with $t\in[0,1]$. Using the Schauder estimates in Appendix~\ref{app: schauder estimates} and arguing as in Lemma~\ref{lem: locally proper full}, $\mathscr{L}^{(t)}$ is also locally proper both $\mathcal{X}_\be\to\mathcal{Y}_\be$ and $\mathcal{X}\to\mathcal{Y}$. By continuity of the index, $\mathscr{L}^{(1)}=\mathscr{F}_w(w,\alpha)$ has the same Fredholm index $0$ as $\mathscr{L}^{(0)}=\mathscr{F}_w(0,\alpha)$ which is invertible by Lemma~\ref{lem: invertible} and hence Fredholm of index $0$. Therefore,  $\mathscr{F}_w(w,\alpha)$ is Fredholm with index $0$ both $\mathcal{X}_\be\to\mathcal{Y}_\be$ and $\mathcal{X}\to\mathcal{Y}$. 
\end{proof}

\subsection{Global continuation}\label{subsect: large-amplitude}
The aim of this subsection is to continue the local curve of solutions $\mathscr{C}_\text{loc}$ in Theorem~\ref{thm: small-amplitude} to obtain a global curve of solutions $\mathscr{C}$ using real-analytic global bifurcation theory. Unfortunately, the classical theorems by Dancer~\cite{dancer} and Buffoni and Toland~\cite{bt:analytic} are not well-suited for our problem. Indeed, the linearized operator $\mathscr F_w(0,\alpha_\mathrm{cr})$ is not Fredholm $\mathcal X \to \mathcal Y$, and we have potential loss of compactness due to the unbounded domain. These issues are dealt with in \cite[Theorem 6.1]{strat}, which has been crafted to specifically suit the case of solitary waves. However, we need a further small modification for the system setting, as our linearized operators are not automatically Fredholm index 0 as soon as they are semi-Fredholm. 

Thankfully, this can be easily resolved by appealing to the continuity of the index. The modified result and its proof are given in Appendix~\ref{app: cited results}. Applied to our present problem, it yields the following.
\begin{theorem}\label{thm: global curve}
	The local curve $\mathscr{C}_\textup{loc}$ is contained in a continuous curve of solutions parameterized as
	\begin{equation*}
		\mathscr{C}=\{(w(s),\alpha(s)):0<s<\infty \}\subset\mathcal{U}
	\end{equation*}
	with the following properties.
  \begin{enumerate}[label=\rm(\alph*)] 
  \item One of two alternatives must hold: either
    \begin{enumerate}[label=\rm(\roman*)] 
    \item \label{thm: global curve:blow-up} as $s\to\infty$,
      \begin{equation*}
        N(s):=\|w(s)\|_\mathcal{X}+\frac{1}{\lambda(w(s),\alpha(s))}+\frac{1}{\alpha(s)}+\frac{1}{\alpha_\crit-\alpha(s)}\longrightarrow\infty;\text{ or }
      \end{equation*}

    \item \label{thm: global curve:cpt} there exists a sequence $s_n\to\infty$ such that $\sup_nN(s_n)<\infty$ but $\{w(s_n)\}$ has no subsequences converging in $\mathcal{X}$.
    \end{enumerate}
  \item Near each point $(w(s_0),\alpha(s))\in\mathscr{C}$, we can re-parameterize $\mathscr{C}$ so that the mapping $s\mapsto(w(s),\alpha(s))$ is real analytic.
  
  \item \label{thm: global local} $(w(s),\alpha(s))\notin\mathscr{C}_\textup{loc}$ for $s$ sufficiently large.
  \end{enumerate}
\end{theorem}
\begin{proof}
	We check that the hypothesis of Theorem~\ref{thm: global bifurcation} are satisfied. In Section~\ref{sect: Functional Analytic} we reformulated our problem \eqref{fullproblem} as a nonlinear operator equation $\mathscr{F}(w,\alpha)=0$ with $\mathscr{F}\colon\mathcal{U}\to\mathcal{Y}$. Clearly, $\mathscr{F}$ is real analytic on $\mathcal{U}$. Theorem~\ref{thm: small-amplitude} constructs a suitable local curve $\mathscr{C}_\text{loc}=\{(w^\ep,\alpha^\ep):0<\ep<\infty \}\subset\mathcal{U}$. From Lemma~\ref{lem: locally proper full}, the linearized operator $\mathscr{F}_w(w,\alpha)$ is locally proper for $(w,\alpha)\in\mathcal{U}$. Finally, from Theorem~\ref{thm: small-amplitude}\ref{invertibility small-amplitude}, $\mathscr{F}_w$ is invertible along $\mathscr{C}_\text{loc}$. 
\end{proof}

\subsection{Proof of the main result}
We now use results from the previous sections to eliminate certain alternatives from Theorem~\ref{thm: global curve}. Recall that the functions $w_1,w_2$ in this theorem are related to $\eta,\zeta$ by \eqref{tilde}. We begin with the loss of compactness alternative \ref{thm: global curve:cpt}.
\begin{lemma}\label{lem: global nodal properties}
	The nodal property \eqref{nodal eta} holds along the global bifurcation curve $\mathscr{C}$.
\end{lemma} 
\begin{proof}
	In order to apply the results from Section~\ref{sect: Nodal Analysis} we must first show that the global curve $\mathscr{C}$ contains no trivial solutions $(0,\alpha)$. By the definition of $\mathcal{U}$, any solutions on $\mathscr{C}$ must have $\alpha<\alpha_\crit$. Consider now the relatively closed set $\mathscr{T}$ of all trivial solutions in $\mathscr{C}$. Clearly, since $\mathscr{C}$ is closed, this set is relatively closed. Moreover, by Lemma~\ref{lem: invertible}, the operator $\mathscr{F}_w(0,\alpha)$ is invertible for all $\alpha<\alpha_\crit$. By the implicit function theorem the trivial solutions must thus lie on locally unique continuous curves parameterized by $w=w(\alpha)$ with $\alpha<\alpha_\crit$, implying that the set $\mathscr{T}$ is relatively open in $\mathscr{C}$. Since $\mathscr{C}$ is continuous, it must be connected and hence we only have two options: either $\mathscr{C}$ contains only trivial solutions, or it contains none. Since $\mathscr{C}_\text{loc}\subset\mathscr{C}$, the former cannot be true. 
	
	From Theorem~\ref{thm: small-amplitude}, the nodal property \eqref{nodal eta} holds along the local curve $\mathscr{C}_\text{loc}$. Now let $\mathscr{N}\subset\mathscr{C}$ denote the set of all $(w,\alpha)\in\mathscr{C}$ satisfying \eqref{nodal eta}. Since $\mathscr{C}$ is connected in $\mathcal{X}\times\mathbb{R}$ and we have shown that there are no trivial solutions in $\mathscr{C}$, Propositions~\ref{prop: closed}, and~\ref{prop: open} yield that $\mathscr{N}\subset\mathscr{C}$ is both relatively open and relatively closed. Since the local curve $\mathscr{C}_\text{loc}\subset\mathscr{N}$ we see that $\mathscr{N}$ is nonempty and therefore we must have $\mathscr{N}=\mathscr{C}$.
\end{proof}
We now get the following result.
\begin{lemma}\label{lem: remove alt ii}
	Alternative \ref{thm: global curve:cpt} in Theorem~\ref{thm: global curve} cannot occur.
\end{lemma}
\begin{proof}
	Lemma~\ref{lem: global nodal properties} ensures that the monotonicity assumption \eqref{monotonicity} holds along the global bifurcation curve $\mathscr{C}$. The statement now follows from Lemma~\ref{lem: compactness}.
\end{proof}
Let us now consider alternative \ref{thm: global curve:blow-up}. More precisely, we wish to deal with the last term.
\begin{lemma}\label{lem: crit F alt}
	If $\|w(s)\|_\mathcal{X}$ and $1/\lambda(w(s),\alpha(s))$ are uniformly bounded along $\mathscr{C}$, then
	\begin{equation*}
		\liminf_{s\rightarrow\infty}\alpha(s)<\alpha_\crit.
	\end{equation*}
\end{lemma}
\begin{proof}
  We argue by contradiction. Suppose that there exists a sequence $s_n\rightarrow\infty$ for which we have $\sup_{n\rightarrow\infty}\|w(s_n)\|_{\mathcal{X}}<\infty$, $\inf_{n\to\infty}(\lambda(w(s_n)),\alpha(s_n))>0$ and $\alpha(s_n)\rightarrow\alpha_\crit$. Applying Lemma~\ref{lem: compactness} allows us to extract a subsequence so that $\{(w(s_n),\alpha(s_n))\}$ converges in $\mathcal{X}\times\mathbb{R}$ to a solution $(w^*,\alpha^*)$ of $\mathscr{F}(w,\alpha)=0$ with $\alpha^*=\alpha_\crit$. Since by continuity $w_1\geq0$ on the surface $\Gamma$, Theorem~\ref{thm: bound on alpha} yields that this must be the trivial solution $w=0$. Hence, we must have $\|w(s_n)\|_\mathcal{X}\to0$. Moreover, from Lemma~\ref{lem: global nodal properties} we know that all $w(s_n)$ satisfy the nodal property \eqref{nodal eta} and thus by Theorem~\ref{thm: small-amplitude}\ref{uniqueness small-amplitude} we must have $(w(s_n),\alpha(s_n))\in\mathscr{C}_\text{loc}$ for $n$ sufficiently large. This is a contradiction to Theorem~\ref{thm: global curve}\ref{thm: global local}. 
\end{proof}

We are now ready to prove the main theorem. Recall from Section~\ref{sec:statement} that our primary interest is in the case $\gamma < 0$; we state the corresponding results for $\gamma \ge 0$ only for completeness.
\begin{theorem}\label{thm: main result full}
	The local curve $\mathscr{C}_\textup{loc}$ is contained in a continuous curve of solutions parameterized as
	\begin{equation*}
	\mathscr{C}=\{(w(s),\alpha(s)):0<s<\infty \}\subset\mathcal{U},
	\end{equation*}
	with the following properties.
	\begin{enumerate}[label=\rm(\alph*)] 
	\item \label{thm: final alternatives} As $s\to\infty$
	\begin{enumerate}[label=\rm(\roman*)] 
		\item \label{thm: negative vorticity} for $\gamma<0$,
		\begin{equation*}
		\frac{1}{\inf_\Gamma(1-2\alpha(s) w_1(s))}+\frac{1}{\inf_\Gamma(w_{1x}^2(s)+(1+w_{1y}(s))^2)}+\frac{1}{\alpha(s)}\longrightarrow\infty.
		\end{equation*}
		\item \label{thm: positive vorticity} for $\gamma > 0$,
		\begin{equation*}
      \sup_\Gamma |\nabla{w_1(s)}|
      +\frac{1}{\inf_\Gamma(w_{1x}^2(s)+(1+w_{1y}(s))^2)}+\frac{1}{\alpha(s)}\longrightarrow\infty;
		\end{equation*}
		\item \label{thm: zero vorticity} for $\gamma = 0$,
		\begin{equation*}
      \frac{1}{\inf_\Gamma(1-2\alpha(s) w_1(s))}\longrightarrow\infty.
		\end{equation*}
	\end{enumerate}
		\item Near each point $(w(s),\alpha(s))\in\mathscr{C}$, we can re-parameterize $\mathscr{C}$ so that the mapping $s\mapsto(w(s),\alpha(s))$ is real analytic.
		\item $(w(s),\alpha(s))\notin\mathscr{C}_\text{loc}$ for $s$ sufficiently large.
    \item Each $(w,s) \in \mathscr C$ satisfies the nodal properties \eqref{nodal eta}.
	\end{enumerate}
\end{theorem}
\begin{proof}
  Combining Lemmas~\ref{lem: global nodal properties} and \ref{lem: remove alt ii} with Theorem~\ref{thm: global curve}, it remains to prove \ref{thm: final alternatives}. Using Lemma~\ref{lem: crit F alt}, alternative \ref{thm: global curve:blow-up} in Theorem~\ref{thm: global curve} is equivalent to
	\begin{align}
    \label{eqn:blowup prelim}
    \|w(s)\|_{\mathcal{X}}
    +\frac{1}{\lambda(w(s),\alpha(s))}+\frac{1}{\alpha(s)}\longrightarrow\infty\text{ as }s\to\infty.
	\end{align}
  Applying the maximum principle and maximum modulus principle to the two factors in the definition \eqref{eqn:lambda} of $\lambda(w,\alpha)$, we see that the second term $1/\lambda(w(s),\alpha(s))$ in \eqref{eqn:blowup prelim} can be bounded above in terms of 	
  \begin{equation*}
		\frac{1}{\inf_\Gamma(1-2\alpha(s) w_1(s))}+\frac{1}{\inf_\Gamma(w_{1x}^2(s)+(1+w_{1y}(s))^2)}.
	\end{equation*}
  By the uniform regularity result Proposition~\ref{prop: uniform regularity}, the first term $\|w(s)\|_{\mathcal X}$ can be bounded in terms of the constant $\delta > 0$ appearing in \eqref{eqn:delta}. Using the maximum principle yet again, we deduce that $\|w(s)\|_{\mathcal X}$ can be bounded in terms of
  \begin{align*}
    \frac{1}{\inf_\Gamma(1-2\alpha(s) w_1(s))}
    +\sup_\Gamma |\nabla{w_1(s)}|
    +\frac{1}{\inf_\Gamma(w_{1x}^2(s)+(1+w_{1y}(s))^2)}.
  \end{align*}
  Putting this all together we find that \eqref{eqn:blowup prelim} implies
	\begin{align}
    \label{eqn:blowup improved}
    \frac{1}{\inf_\Gamma(1-2\alpha(s) w_1(s))}
    +\sup_\Gamma |\nabla{w_1(s)}|
    +\frac{1}{\inf_\Gamma(w_{1x}^2(s)+(1+w_{1y}(s))^2)}
    +\frac{1}{\alpha(s)}\longrightarrow\infty\text{ as }s\to\infty.
	\end{align}
  We will now further simplify \eqref{eqn:blowup improved} based on the sign of $\gamma$.

  For $\gamma < 0$, we eliminate the second term in \eqref{eqn:blowup improved}. Consider a fixed solution $(w,\alpha) \in \mathscr C$. By Proposition~\ref{prop: gamma cases}\ref{gamma negative}, the stream function $\psi$ defined in \eqref{definition w} satisfies the upper bound
	\begin{equation}\label{new gamma negative}
    \psi_y<1-\tfrac 12 \gamma\qquad\text{ on }\Gamma.
	\end{equation}
  We claim that $\psi_y > 0$ on $\Gamma$. Assuming the claim, \eqref{new gamma negative} and \eqref{eta surface} yield
	\begin{equation}\label{dynamic recall}
		(1-2\alpha w_1) (w_{1x}^2(s)+(1+w_{1y}(s))^2)
    =\psi_y^2
    \le (1-\tfrac 12 \gamma)^2
    \qquad\text{on }\Gamma,
	\end{equation}
  so that the second term in \eqref{eqn:blowup improved} is bounded above by a multiple of the third term, proving \ref{thm: negative vorticity}. To prove the claim, we simply note that $\psi_y^2 \ge \lambda(w,\alpha) > 0$ on $\Gamma$, where here we have used the first equality in \eqref{dynamic recall}, the definition \eqref{eqn:lambda} of $\lambda(w,s)$ and the fact that $(w,\alpha) \in \mathscr C \subset \mathcal U$. As the asymptotic conditions for $w$ imply $\psi_y(x,1) \to 1$ as $x \to \pm\infty$, the claim then follows by continuity.

  For $\gamma \ge 0$, a similar argument using Proposition~\ref{prop: gamma cases}\ref{gamma positive} instead yields the lower bound
	\begin{equation}\label{dynamic recall again}
		1-2\alpha w_1 
    \ge 
    \frac{
    \Big(\min\big\{2-\gamma,\gamma\inf_{\mathcal{R}}
    (w_{1x}^2+(1+w_{1y})^2)
     \big\}\Big)^2}
   {w_{1x}^2+(1+w_{1y})^2}
    \qquad\text{on }\Gamma.
	\end{equation}
  Thus the first term in \eqref{eqn:blowup improved} is controlled by the second two, proving \ref{thm: positive vorticity}. We note that in many cases the last term in \eqref{eqn:blowup improved} can also be controlled; see~\cite{froude,strat}.

  Finally, we turn to the irrotational case $\gamma=0$. As for $\gamma < 0$, we can control the second term in \eqref{eqn:blowup improved} using the third term. On the other hand, the numerator in \eqref{dynamic recall again} simplifies and we see that the third term in \eqref{eqn:blowup improved} is bounded by a multiple of the first term. Well-known bounds on the Froude number for irrotational solitary waves yield $\alpha > 1/4$~\cite{at:finite,mcleod}, or indeed the stronger bound $\alpha > 1/2$~\cite{kp}, eliminating the last term in \eqref{eqn:blowup improved} and hence proving \ref{thm: zero vorticity}.
\end{proof}

We now show how the above result for negative vorticity implies Theorem~\ref{thm: main result}.

\begin{proof}[Proof of Theorem~\ref{thm: main result}]
  We first claim that the solutions $(w,\alpha)$ along the curve $\mathscr C$ in Theorem~\ref{thm: main result full} correspond to solutions of the original physical problem \eqref{originalproblem} or equivalently \eqref{euler}. Reversing the arguments in Section~\ref{subsect: conformal map}, it remains only to show that $\eta$ is the imaginary part of a conformal mapping $\xi+i\eta$ defined on the infinite strip $\mathcal R$. As $\eta$ is harmonic, we can easily define $\xi$ using the Cauchy--Riemman equations. By the Darboux--Picard theorem \cite[Corollary 9.16]{burckel:book} it is then sufficient to show that $\xi+i\eta$ is injective on the boundary $\partial \mathcal R = \Gamma \cup B$. By construction we have $\eta=0$ on $B$, and the nodal properties \eqref{nodal eta} imply that $\eta > 0$ in $\Gamma$. Thus the images of $\Gamma$ and $B$ do not intersect. Applying the Hopf lemma to $\eta$ on $B$ we discover that $\eta_y = \xi_x > 0$ there, and so it enough to consider the restriction of $\xi+i\eta$ to the surface $\Gamma$. Suppose for the sake of contradiction that this restriction is not injective. Using the evenness of $\eta$ and the nodal properties \eqref{nodal eta}, we easily check that $\xi$ achieves its nonpositive infimum over the half-strip $\mathcal R^+$ at a point $(x_0,1) \in \Gamma^+$. By the Hopf lemma, we therefore have $\xi_y = \eta_x < 0$ at this point, contradicting the nodal properties \eqref{nodal eta}.

	Recall that we switched to dimensionless variables in Section~\ref{subsect: nondim}. Using stars to denote the associated dimensional quantities, we have $\alpha=1/F^2$ and
  \begin{align*}
    1-2\alpha(\eta^*-d)/d
    &=
    1-2\alpha (\eta-1)  
    = 1-2\alpha w_1,\\
    |\nabla\eta^*(x^*,y^*)|^2
    &= |\nabla\eta(x,y)|^2 
    = w_{1x}^2(x,y)
    + (1+w_{1y}(x,y))^2.
	\end{align*}
  Thus \eqref{main result} in Theorem~\ref{thm: main result} follows directly from \ref{thm: negative vorticity} in Theorem~\ref{thm: main result full}.
\end{proof}
Theorem~\ref{thm: main result} only considers the case of constant negative vorticity. However, in Theorem~\ref{thm: main result full} we also have a result for $\gamma>0$ and for the irrotational case $\gamma=0$. We provide those results in physical variables below. The proof the almost identical to that of Theorem~\ref{thm: main result} and is hence omitted.

\begin{proposition}\label{prop: positive gamma result}
	Fix the gravitational constant $g>0$, asymptotic depth $d>0$ and $\gamma\geq0$. Then there exists a global continuous curve $\mathscr{C}$ of solutions to \eqref{euler} parameterized by $s$, with $s\in(0,\infty)$. Moreover, the following property holds along $\mathscr{C}$ as $s\to\infty$:
	\begin{enumerate}[label=\rm(\roman*)]
		\item for $\gamma>0$
			\begin{equation*}
		\min\bigg\{
        \inf_\Gamma|\nabla\eta(s)|,\,
         \bigg(\sup_\Gamma|\nabla\eta(s)|\bigg)^{-1},\,
        \frac{1}{F(s)}
        \bigg\}\longrightarrow0.
		\end{equation*}
		\item for $\gamma=0$
		\begin{equation*}
		\inf_\Gamma\bigg(1-2\alpha\frac{\eta-d}{d}\bigg)\longrightarrow0.
		\end{equation*}
	\end{enumerate}
  These solutions are all symmetric and monotone waves of elevation in the sense that $\eta$ is even in $x$ with $\eta_x(x,d)<0$ for $x>0$.	
\end{proposition}

\section*{Acknowledgements}
The work of SVH is funded by the National Science Foundation through the award DMS-2102961.

\appendix

\section{Schauder estimates for elliptic systems}\label{app: schauder estimates}
Let us define the rectangular strip
\begin{equation*}
	\mathcal{R}=\{(x,y)\in\mathbb{R}^2:0<y<1 \}
\end{equation*}
with top boundary
\begin{equation*}
	\Gamma=\{(x,y)\in\mathbb{R}^2:y=1\}.
\end{equation*}
Moreover, we only work with functions $w_i$ which vanish on the bottom boundary, $y=0$. Consider linear elliptic problems of the form
\begin{subequations}\label{linear elliptic problem}
	\begin{alignat}{2}
		\label{schauder harmonic v1}	
		\Delta w_1&=0&\qquad&\text{in }\mathcal{R},\\
		\label{schauder harmonic v2}
		\Delta w_2&=0&\qquad&\text{in }\mathcal{R},\\
		\label{schauder dynamic}
		a_{ij}\partial_jw_i+b_iw_i&=f&\qquad&\text{on }\Gamma,\\
		\label{schauder kinematic}
		c_iw_i&=g&\qquad&\text{on }\Gamma,
	\end{alignat}
\end{subequations}
where here $\partial_1=\partial_x$, $\partial_2=\partial_y$ and we are using the summation convention. We will now provide sufficient conditions on the coefficients in \eqref{schauder dynamic}--\eqref{schauder kinematic} such that we have access to Schauder estimates for \eqref{linear elliptic problem}. In order to achieve this, we follow the procedures outlined in \cite{volpert:book} and \cite{wloka:book}. We obtain the following lemma.
\begin{lemma}\label{lem: schauder estimates}
	Fix $k\geq1$ and $\beta\in(0,1)$, and suppose that the coefficients in \eqref{linear elliptic problem} have the regularity $a_{ij}, b_i\in C_\be^{k+\beta}(\Gamma)$ and $c_i\in C_\be^{k+1+\beta}(\Gamma)$. If the uniform bound
	\begin{equation}\label{complementing condition bound}
		(c_1a_{21}-c_2a_{11})^2+(c_1a_{22}-c_2a_{12})^2\geq\lambda
	\end{equation} 
	holds for some constant $\lambda>0,$ then \eqref{linear elliptic problem} enjoys the Schauder estimate
	\begin{equation*}\label{schauder estimate}
		\|w_1\|_{C^{k+1+\beta}(\mathcal{R})}+\|w_2\|_{C^{k+1+\beta}(\mathcal{R})}\leq C\big(\|f\|_{C^{k+\beta}(\Gamma)}+\|g\|_{C^{k+1+\beta}(\Gamma)}+\|w_1\|_{C^0(\mathcal{R})}+\|w_2\|_{C^0(\mathcal{R})} \big),
	\end{equation*}
	where the constant $C>0$ depends only on $k,\beta,\lambda$ and on the stated norms of the coefficients.
\end{lemma}
\begin{remark}
	The constant $\lambda>0$ in \eqref{complementing condition bound} is referred to as the ``minor constant'' in \cite{adn2}. 
\end{remark}
\begin{proof}
	To begin with, \eqref{schauder harmonic v1}--\eqref{schauder harmonic v2} can be written as 
	\begin{equation*}
	\mathscr{L}(\partial)w:=
	\begin{pmatrix}
	\mathscr{L}_{11}(\partial) & \mathscr{L}_{12}(\partial) \\\mathscr{L}_{21}(\partial) & \mathscr{L}_{22}(\partial) 
	\end{pmatrix}w
	=\begin{pmatrix}
	\partial_1^2+\partial_2^2 & 0\\0 &\partial_1^2+\partial_2^2
	\end{pmatrix}
	\begin{pmatrix}
	w_1\\w_2
	\end{pmatrix}
	=0.
	\end{equation*}
	Here we think of $\mathscr{L}_{ij}(\xi)=\delta_{ij}(\xi_1^2+\xi_2^2)$ as polynomials in $\xi\in\mathbb{C}^2$. Moreover, we attach integer weights $s_1=s_2=0$ to the equation and $t_1=t_2=2$ to the unknowns. These weights satisfy the constraint 
	\begin{equation*}
	\operatorname{deg}\mathscr{L}_{ij}(\xi)\leq s_i+t_j\qquad\text{for }i,j=1,2.
	\end{equation*}
	Clearly $\mathscr{L}$ is uniformly elliptic in the sense that there exists some positive constant $A$ such that
	\begin{equation*}\label{uniform ellipticiy}
	A^{-1}|\xi|^4\leq|\mathscr{L}(\xi)|\leq A|\xi|^4
	\end{equation*}
	for every real vector $\xi\in\mathbb{R}^2$. 
	
	We now express the boundary conditions \eqref{schauder dynamic}--\eqref{schauder kinematic} as
	\begin{equation*}
	\tilde{\mathscr{B}}(x,\partial)w:=
	\begin{pmatrix}
	a_{11}(x)\partial_1+a_{12}(x)\partial_2+b_1(x) & a_{21}(x)\partial_1+a_{22}(x)\partial_2+b_2(x)\\ c_1(x) & c_2(x)
	\end{pmatrix}	
	\begin{pmatrix}
	w_1\\w_2
	\end{pmatrix}=0,
	\end{equation*}
	for $x\in\Gamma$. Here, as above, we consider $\tilde{\mathscr{B}}_{hj}=\tilde{\mathscr{B}}_{hj}(x,\xi)$ as polynomials in $\xi\in\mathbb{C}^2$. We again assign the integer weights $t_1=t_2=2$ to the dependent variables as well as the weight $r_1=-2$ and $r_2=-1$ to the equations such that
	\begin{equation*}
		\operatorname{deg}\tilde{\mathscr{B}}_{hj}\leq r_h+t_j.
	\end{equation*}
	From this point on, we will only consider the principle boundary operator $\mathscr{B}(x,\xi)$. This operator only consists of the terms of $\tilde{\mathscr{B}}(x,\xi)$ which are of order $r_h+t_j$. That is,
	\begin{equation*}\label{boundary op}
		\mathscr{B}(x,\xi)=
		\begin{pmatrix}
		a_{11}(x,\xi)+a_{12}(x,\xi) & a_{21}(x,\xi)+a_{22}(x,\xi)\\ c_1(x) & c_2(x)
		\end{pmatrix}.
	\end{equation*}  
	Let us now fix the point on the boundary $x\in\Gamma$ and write  
	\begin{equation*}\label{tau nu map}
		\xi=(\tau,\nu),
	\end{equation*} 
	with $\tau$ and $\nu$ denoting the tangential and normal components to $\Gamma$, respectively. In this notation, we have
	\begin{align*}\label{L tau nu}
		\mathscr{L}(\tau,\nu)&=
		\begin{pmatrix}
		\tau^2+\nu^2 & 0\\ 0 & \tau^2+\nu^2
		\end{pmatrix},\\
		\mathscr{B}(x,\tau,\nu)&=
		\begin{pmatrix}
		a_{11}(x)\tau+a_{12}(x)\nu & a_{21}(x)\tau+a_{22}(x)\nu\\
		c_1 & c_2
		\end{pmatrix}.
	\end{align*}
	
	We now construct the so-called Lopatinskii matrix; see \cite{volpert:book}.  This $2\times4$ matrix is defined as the line integral 
	\begin{equation}\label{lopatinskii matrix}
		\Lambda(x,\tau):=\int_{\gamma}\Big(\mathscr{B}(x,\tau,\mu)\,\mathscr{L}^{-1}(\tau,\mu)\quad\mu\,\mathscr{B}(x,\tau,\mu)\,\mathscr{L}^{-1}(\tau,\mu)\Big)\,d\mu,
	\end{equation}
	where the line integral is taken over any Jordan curve $\gamma$ which lies in the upper half-plane $\operatorname{Im}\mu>0$ and encloses the roots of $\operatorname{det}\mathscr{L}(\tau,\mu)=\tau^2+\mu^2$ with positive imaginary part. Calculating the integrand in \eqref{lopatinskii matrix} explicitly yields
	\begin{equation*}\label{lopatinskii integrand}
		\Big(\mathscr{B}\,\mathscr{L}^{-1}\quad\mu\,\mathscr{B}\,\mathscr{L}^{-1}\Big)=\frac{1}{\tau^2+\mu^2}
		\begin{pmatrix}
		a_{11}\tau+a_{12}\mu & a_{21}\tau+a_{22}\mu & \mu(a_{11}\tau+a_{12}\mu) & \mu(a_{21}\tau+a_{22}\mu)\\
		c_1 & c_2 & \mu c_1 & \mu c_2
		\end{pmatrix}.
	\end{equation*}
    The only relevant root is $\mu=i\tau$. Therefore, using calculus of residues, we calculate
	\begin{equation*}
		\Lambda(x,\tau)=\frac{1}{2i\tau}
		\begin{pmatrix}
		a_{11}\tau+ia_{12}\tau & a_{21}\tau+ia_{22}\tau & -a_{12}\tau+ia_{11}\tau^2 & -a_{22}\tau+ia_{21}\tau^2\\
		c_1 & c_2 & ic_1\tau & ic_2\tau
		\end{pmatrix}.
	\end{equation*}
	The matrix \eqref{lopatinskii matrix} thus has rank $2$ in a uniform sense provided  
	\begin{equation}\label{minor condition}
		\inf_{x\in\Gamma,\; |\tau|=1}\big(|M_1|+\cdots+|M_6| \big)>0,
	\end{equation}
	where $M_1,\ldots,M_6$ denote all the minor determinants of $\Lambda(x,\tau,\nu)$. A calculation shows that each $|M_i|$ for $i=1,\ldots,6$ is equal to either $|a_{11}c_2-a_{21}c_1|$ or $|a_{12}c_2-a_{22}c_1|$. Therefore, \eqref{minor condition} is equivalent to 
	\begin{equation*}
		\inf_{x\in\Gamma,\; |\tau|=1}\big( (c_1a_{21}-c_2a_{11})^2+(c_1a_{22}-c_2a_{12})^2 \big)\geq\lambda,
	\end{equation*}  
	for some $\lambda>0$. The result now follows from \cite[Theorem 9.2]{adn2}.
\end{proof}

\section{Abstract global bifurcation theorem}\label{app: cited results}
In this section, we provide an abstract global bifurcation theorem for real-analytic operators. This is a modified version a global bifurcation theorem for solitary waves, as stated in \cite[Theorem 6.1]{strat}, which is itself a modification of the results by Dancer~\cite{dancer} and Buffoni and Toland~\cite{bt:analytic}. A very similar result also appears in~\cite[Theorem~B.1]{cww:globalfronts}. The only difference is that in \ref{semifredholm} we allow for the linearized operators to be locally proper rather than requiring them to be Fredholm of index $0$.
\begin{theorem}\label{thm: global bifurcation}
	Let $\mathscr{X}$ and $\mathscr{Y}$ be Banach spaces, $\mathscr{U}$ be an open subset of $\mathscr{X}\times\mathbb{R}$ with $(0,0)\in\partial\mathscr{U}$. Consider a real-analytic mapping $\mathcal{F}\colon\mathscr{U}\to\mathscr{Y}$. Suppose that
	\begin{enumerate}[label=\rm(\Roman*)]
		\item \label{semifredholm} for any $(\mu,x)\in\mathscr{U}$ with $\mathcal{F}(\mu,x)=0$ the Fréchet derivative $\mathcal{F}_x(\mu,x)\colon\mathscr{X}\to\mathscr{Y}$ is locally proper;
		\item there exists a continuous curve $\mathscr{C}_\textup{loc}$ of solutions to $\mathcal{F}(\mu,x)=0,$ parameterized as 
		\begin{equation*}
		\mathscr{C}_{\textup{loc}}:=\{(\mu,\tilde{x}(\mu)):0<\mu<\mu_* \}\subset\mathcal{F}^{-1}(0),
		\end{equation*}
		for some $\mu_*>0$ and continuous $\tilde{x}$ with values in $\mathscr{X}$ and $\lim_{\mu\searrow0}\tilde{x}(\mu)=0$;
		\item \label{invertible} the linearized operator $\mathcal{F}_x(\mu,\tilde{x}(\mu))\colon\mathscr{X}\to\mathscr{Y}$ is invertible for all $\mu$.
	\end{enumerate}
	Then $\mathscr{C}_\text{loc}$ is contained is a curve of solutions $\mathscr{C}$,  parameterized as
	\begin{equation*}
		\mathscr{C}:=\{(\mu(s),x(s)):0<s<\infty \}\subset\mathcal{F}^{-1}(0)
	\end{equation*}
	for some continuous $(0,\infty)\ni s\mapsto(x(s),\mu(s) )\in\mathscr{U},$ with the following properties.
	\begin{enumerate}[label=\rm(\alph*)]
	\item One of the following alternatives holds:
		\begin{enumerate}[label=\rm(\roman*)]
			\item \label{blow-up}  \textup{(Blow-up)} As $s\to\infty$,
	\begin{equation*}
		N(s):=\|x(s)\|_\mathscr{X}+\frac{1}{\operatorname{dist}((\mu(s),x(s)),\partial\mathscr{U})}+\mu(s)\to\infty.
	\end{equation*}
	
			\item \label{loss of compactness} \textup{(Loss of compactness)} There exists a sequence $s_n\to\infty$ such that $\operatorname{sup}_n N(s_n)<\infty$ but $\{x(s_n)\}$ has no subsequences converging in $\mathscr{X}$.
		\end{enumerate}
	\item  Near each point $(\mu(s_0),x(s_0))\in\mathscr{C}$, we can reparameterize $\mathscr{C}$ so that $s\mapsto(\mu(s),x(s))$ is real analytic.
	
	\item $(\mu(s),x(s))\notin\mathscr{C}_{\text{loc}}$ for $s$ sufficiently large.
	\end{enumerate}
\end{theorem}
\begin{proof}
	The proof of the theorem is almost identical to the proof of \cite[Theorem 6.1]{strat}, and so we only give a brief sketch. As \cite{strat}, since \ref{invertible} holds we can construct the distinguished arc $A_0$, the connected component of
	\begin{equation*}
		\mathcal{A}:=\big\{(\mu,x)\in\mathscr{U}:\mathcal{F}(\mu,x)=0,\; \mathcal{F}_x(\mu,x)\text{ is invertible }  \big\}
	\end{equation*}
	in which $(\mu_{1/2},x_{1/2}):=(\mu_*/2,x(\mu_*/2))$ lies. The analytic implicit function theorem guarantees that all distinguished arcs are graphs. After possibly re-parameterizing, we write $A_0$ as
	\begin{equation*}
		A_0=\{(\mu(s),x(s)):0<s<1 \},
	\end{equation*}
	where $\mu(s)$ is increasing. From the implicit function theorem, the local curve of solutions $\mathscr{C}_\text{loc}$ lies entirely in $A_0$. Arguing as in the proof of \cite[Theorem 6.1]{strat}, the starting point of $A_0$ is
	\begin{equation*}
		\lim_{s\searrow0}(\mu(s),x(s))=(0,0).
	\end{equation*}
	The next step is to consider the limit $s\nearrow1$. As in the proof of \cite[Theorem 6.1]{strat} we now have two options. Either, $N(s)\to\infty$ as $s\nearrow1$ in which case after re-parametrization \ref{blow-up} occurs, or there exists some sequence $\{s_n\}\subset\big(\tfrac{1}{2},1\big)$ with $s_n\nearrow1$ so that $N(s_n)\leq M<\infty$ for all $n>0$. Assuming without loss of generality that $\mu(s_n)\to\mu_{1}$, we now again have two possibilities. Either $\{x(s_n)\}$ has no convergent subsequence, in which case alternative \ref{loss of compactness} takes place, or, after extraction, $(\mu(s_n),x(s_n))\to (\mu_1,x_1)\in\mathscr{U}$. 
	
	By continuity, $\mathcal{F}(\mu_{1},x_1)=0$. Moreover, by assumption the linearized operator $\mathcal{F}_x(\mu,x)$ is Fredholm with index 0 along $\mathscr{C}_\text{loc}$ and semi-Fredholm beyond. By continuity of the index, we can now conclude that $\mathcal{F}_x(\mu_1,x_1)$ is Fredholm index $0$. The rest of the proof, establishing the existence of an infinite sequence of connected distinguished arcs $A_n$ along which either \ref{blow-up} or \ref{loss of compactness} must hold, now follows exactly as in the proof of \cite[Theorem 6.1]{strat}.
\end{proof}

\bibliographystyle{amsalpha}
\bibliography{references}

\newcommand{\etalchar}[1]{$^{#1}$}
\renewcommand{\MR}[1]{}\def\cprime{$'$}
\providecommand{\bysame}{\leavevmode\hbox to3em{\hrulefill}\thinspace}
\providecommand{\MR}{\relax\ifhmode\unskip\space\fi MR }
\providecommand{\MRhref}[2]{%
  \href{http://www.ams.org/mathscinet-getitem?mr=#1}{#2}
}
\providecommand{\href}[2]{#2}
\begin{thebibliography}{HHS{\etalchar{+}}21}

\bibitem[ADN64]{adn2}
S.~Agmon, A.~Douglis, and L.~Nirenberg, \emph{Estimates near the boundary for
  solutions of elliptic partial differential equations satisfying general
  boundary conditions. {II}}, Comm. Pure Appl. Math. \textbf{17} (1964),
  35--92. \MR{162050}

\bibitem[AFT82]{aft}
Charles~J. Amick, L.~E. Fraenkel, and J.~F. Toland, \emph{On the {S}tokes
  conjecture for the wave of extreme form}, Acta Math. \textbf{148} (1982),
  193--214. \MR{666110 (83m:35147)}

\bibitem[Ami87]{amick:bounds}
Charles~J. Amick, \emph{Bounds for water waves}, Arch. Rational Mech. Anal.
  \textbf{99} (1987), no.~2, 91--114. \MR{886932 (88i:76009)}

\bibitem[AT81a]{at:periodic}
Charles~J. Amick and J.~F. Toland, \emph{On periodic water-waves and their
  convergence to solitary waves in the long-wave limit}, Philos. Trans. Roy.
  Soc. London Ser. A \textbf{303} (1981), no.~1481, 633--669. \MR{647410
  (83b:76009)}

\bibitem[AT81b]{at:finite}
\bysame, \emph{On solitary water-waves of finite amplitude}, Arch. Rational
  Mech. Anal. \textbf{76} (1981), no.~1, 9--95. \MR{629699 (83b:76017)}

\bibitem[Bab87]{babenko}
K.~I. Babenko, \emph{Some remarks on the theory of surface waves of finite
  amplitude}, Dokl. Akad. Nauk SSSR \textbf{294} (1987), no.~5, 1033--1037.
  \MR{898306}

\bibitem[Bea77]{beale}
J.~Thomas Beale, \emph{The existence of solitary water waves}, Comm. Pure Appl.
  Math. \textbf{30} (1977), no.~4, 373--389. \MR{0445136 (56 \#3480)}

\bibitem[Ben84]{benjamin:impulse}
T.~Brooke Benjamin, \emph{Impulse, flow force and variational principles}, IMA
  J. Appl. Math. \textbf{32} (1984), no.~1-3, 3--68. \MR{740456 (85h:70012a)}

\bibitem[BT03]{bt:analytic}
Boris Buffoni and John Toland, \emph{Analytic theory of global bifurcation},
  Princeton Series in Applied Mathematics, Princeton University Press,
  Princeton, NJ, 2003, An introduction. \MR{1956130 (2004b:47117)}

\bibitem[Bur79]{burckel:book}
Robert~B. Burckel, \emph{An introduction to classical complex analysis. {V}ol.
  1}, Pure and Applied Mathematics, vol.~82, Academic Press, Inc. [Harcourt
  Brace Jovanovich, Publishers], New York-London, 1979. \MR{555733}

\bibitem[Cra04]{craik}
Alex D.~D. Craik, \emph{The origins of water wave theory}, Annual review of
  fluid mechanics. {V}ol. 36, Annu. Rev. Fluid Mech., vol.~36, Annual Reviews,
  Palo Alto, CA, 2004, pp.~1--28. \MR{2062306 (2005a:01012)}

\bibitem[CS88]{cs:sym}
Walter Craig and Peter Sternberg, \emph{Symmetry of solitary waves}, Comm.
  Partial Differential Equations \textbf{13} (1988), no.~5, 603--633.
  \MR{919444 (88m:35132)}

\bibitem[CS04]{cs:exact}
Adrian Constantin and Walter~A. Strauss, \emph{Exact steady periodic water
  waves with vorticity}, Comm. Pure Appl. Math. \textbf{57} (2004), no.~4,
  481--527. \MR{2027299 (2004i:76018)}

\bibitem[CS07]{cs:stag}
\bysame, \emph{Rotational steady water waves near stagnation}, Philos. Trans.
  R. Soc. Lond. Ser. A Math. Phys. Eng. Sci. \textbf{365} (2007), no.~1858,
  2227--2239. \MR{2329144 (2008j:76013)}

\bibitem[CSV16]{csv:critical}
Adrian Constantin, Walter Strauss, and Eugen V\u{a}rv\u{a}ruc\u{a},
  \emph{Global bifurcation of steady gravity water waves with critical layers},
  Acta Math. \textbf{217} (2016), no.~2, 195--262. \MR{3689941}

\bibitem[CSV21]{csv:downstream}
\bysame, \emph{Large-{A}mplitude {S}teady {D}ownstream {W}ater {W}aves}, Comm.
  Math. Phys. \textbf{387} (2021), no.~1, 237--266. \MR{4312364}

\bibitem[CV11]{cv:constant}
Adrian Constantin and Eugen Varvaruca, \emph{Steady periodic water waves with
  constant vorticity: regularity and local bifurcation}, Arch. Ration. Mech.
  Anal. \textbf{199} (2011), no.~1, 33--67. \MR{2754336 (2012e:76017)}

\bibitem[CWW18]{strat}
Robin~Ming Chen, Samuel Walsh, and Miles~H. Wheeler, \emph{Existence and
  qualitative theory for stratified solitary water waves}, Ann. Inst. H.
  Poincar\'e Anal. Non Lin\'eaire \textbf{35} (2018), no.~2, 517--576.
  \MR{3765551}

\bibitem[CWW19]{chen2019center}
\bysame, \emph{Center manifolds without a phase space for quasilinear problems
  in elasticity, biology, and hydrodynamics}, arXiv preprint arXiv:1907.04370
  (2019).

\bibitem[CWW20]{cww:globalfronts}
\bysame, \emph{Global bifurcation for monotone fronts of elliptic equations},
  arXiv preprint arXiv:2005.00651 (2020).

\bibitem[Dan73]{dancer}
E.~N. Dancer, \emph{Bifurcation theory for analytic operators}, Proc. London
  Math. Soc. (3) \textbf{26} (1973), 359--384. \MR{322615}

\bibitem[Dar03]{darrigol:horse}
Olivier Darrigol, \emph{The spirited horse, the engineer, and the
  mathematician: water waves in nineteenth-century hydrodynamics}, Arch. Hist.
  Exact Sci. \textbf{58} (2003), no.~1, 21--95. \MR{2020055 (2004k:76002)}

\bibitem[DH19a]{dh:foldsgaps}
Sergey~A. Dyachenko and Vera~Mikyoung Hur, \emph{Stokes waves with constant
  vorticity: folds, gaps and fluid bubbles}, J. Fluid Mech. \textbf{878}
  (2019), 502--521. \MR{4010456}

\bibitem[DH19b]{dh:numerical}
\bysame, \emph{Stokes waves with constant vorticity: {I}. {N}umerical
  computation}, Stud. Appl. Math. \textbf{142} (2019), no.~2, 162--189.
  \MR{3915685}

\bibitem[DJ34]{dubreil}
M.~L. Dubreil-Jacotin, \emph{Sur la d\'etermination rigoureuse des ondes
  permanentes p\'eriodiques d'amplitude finie}, Journ. de Math. \textbf{13}
  (1934), 217--289.

\bibitem[FH54]{fh}
K.~O. Friedrichs and D.~H. Hyers, \emph{The existence of solitary waves}, Comm.
  Pure Appl. Math. \textbf{7} (1954), 517--550. \MR{0065317 (16,413f)}

\bibitem[GT01]{GT}
David Gilbarg and Neil~S. Trudinger, \emph{Elliptic partial differential
  equations of second order}, Classics in Mathematics, Springer-Verlag, Berlin,
  2001, Reprint of the 1998 edition. \MR{1814364 (2001k:35004)}

\bibitem[GW08]{gw}
Mark~D. Groves and Erik Wahl{\'e}n, \emph{Small-amplitude {S}tokes and solitary
  gravity water waves with an arbitrary distribution of vorticity}, Phys. D
  \textbf{237} (2008), no.~10-12, 1530--1538. \MR{2454604 (2011a:37122)}

\bibitem[HHS{\etalchar{+}}21]{onepas}
Susanna~V. Haziot, Vera~Mikyoung Hur, Walter Strauss, J.~F. Toland, Erik
  Wahlén, Samuel Walsh, and Miles~H. Wheeler, \emph{Traveling water waves --
  the ebb and flow of two centuries}, arXiv preprint arXiv:2109.09208 (2021).

\bibitem[Hur08a]{hur:exact}
Vera~Mikyoung Hur, \emph{Exact solitary water waves with vorticity}, Arch.
  Ration. Mech. Anal. \textbf{188} (2008), no.~2, 213--244. \MR{2385741
  (2010b:76019)}

\bibitem[Hur08b]{hur:symmetry}
\bysame, \emph{Symmetry of solitary water waves with vorticity}, Math. Res.
  Lett. \textbf{15} (2008), no.~3, 491--509. \MR{2407226 (2009c:35130)}

\bibitem[HW21]{hw:touching}
Vera~Mikyoung Hur and Miles~H. Wheeler, \emph{Overhanging and touching waves in
  constant vorticity flows}, arXiv preprint arXiv:2107.14014 (2021).

\bibitem[Kir88]{kirch:res}
Klaus Kirchg{\"a}ssner, \emph{Nonlinearly resonant surface waves and homoclinic
  bifurcation}, Advances in applied mechanics, {V}ol.\ 26, Adv. Appl. Mech.,
  vol.~26, Academic Press, Boston, MA, 1988, pp.~135--181. \MR{1076007}

\bibitem[KKL20]{kkl:stag}
V.~Kozlov, N.~Kuznetsov, and E.~Lokharu, \emph{Solitary waves on constant
  vorticity flows with an interior stagnation point}, J. Fluid Mech.
  \textbf{904} (2020), A4, 18. \MR{4164801}

\bibitem[KLW21]{klw:subcritical}
Vladimir Kozlov, Evgeniy Lokharu, and Miles~H. Wheeler, \emph{Nonexistence of
  subcritical solitary waves}, Arch. Ration. Mech. Anal. \textbf{241} (2021),
  no.~1, 535--552. \MR{4271966}

\bibitem[KP74]{kp}
G.~Keady and W.~G. Pritchard, \emph{Bounds for surface solitary waves}, Proc.
  Cambridge Philos. Soc. \textbf{76} (1974), 345--358. \MR{0347216 (49
  \#11936)}

\bibitem[Lav54]{lavrentiev}
M.~A. Lavrentiev, \emph{I. {O}n the theory of long waves. {II}. {A}
  contribution to the theory of long waves}, Amer. Math. Soc. Translation
  \textbf{1954} (1954), no.~102, 53. \MR{0061952 (15,906a)}

\bibitem[Lie87]{lieberman}
Gary~M. Lieberman, \emph{Two-dimensional nonlinear boundary value problems for
  elliptic equations}, Trans. Amer. Math. Soc. \textbf{300} (1987), no.~1,
  287--295. \MR{871676}

\bibitem[McL84]{mcleod}
J.~B. McLeod, \emph{The {F}roude number for solitary waves}, Proc. Roy. Soc.
  Edinburgh Sect. A \textbf{97} (1984), 193--197. \MR{751191 (85g:76011)}

\bibitem[Mie88]{mielke}
Alexander Mielke, \emph{Reduction of quasilinear elliptic equations in
  cylindrical domains with applications}, Math. Methods Appl. Sci. \textbf{10}
  (1988), no.~1, 51--66. \MR{929221 (89d:35063)}

\bibitem[Pom92]{pommerenke:book}
Ch. Pommerenke, \emph{Boundary behaviour of conformal maps}, Grundlehren der
  Mathematischen Wissenschaften [Fundamental Principles of Mathematical
  Sciences], vol. 299, Springer-Verlag, Berlin, 1992. \MR{1217706}

\bibitem[Rus44]{russell}
J.~Scott Russell, \emph{Report on waves}, 14th meeting of the British
  Association for the Advancement of Science, vol. 311--390, 1844.

\bibitem[Ser71]{serrin}
James Serrin, \emph{A symmetry problem in potential theory}, Arch. Rational
  Mech. Anal. \textbf{43} (1971), 304--318. \MR{333220}

\bibitem[SS85]{ss:deep}
J.~A. Simmen and P.~G. Saffman, \emph{Steady deep-water waves on a linear shear
  current}, Stud. Appl. Math. \textbf{73} (1985), no.~1, 35--57. \MR{797557}

\bibitem[TdSP88]{sp:steep}
A.~F. Teles~da Silva and D.~H. Peregrine, \emph{Steep, steady surface waves on
  water of finite depth with constant vorticity}, J. Fluid Mech. \textbf{195}
  (1988), 281--302. \MR{985439 (90a:76061)}

\bibitem[TK60]{ter:orig}
A.~M. Ter-Krikorov, \emph{The existence of periodic waves which degenerate into
  a solitary wave}, Journal of Applied Mathematics and Mechanics \textbf{24}
  (1960), no.~4, 930--949.

\bibitem[TK61]{ter:rot}
\bysame, \emph{The solitary wave on the surface of a turbulent liquid}, \v Z.
  Vy\v cisl. Mat. i Mat. Fiz. \textbf{1} (1961), 1077--1088. \MR{0145776 (26
  \#3304)}

\bibitem[Var20]{varholm:global}
Kristoffer Varholm, \emph{Global bifurcation of waves with multiple critical
  layers}, SIAM J. Math. Anal. \textbf{52} (2020), no.~5, 5066--5089.
  \MR{4164492}

\bibitem[VB94]{vandenb:constant}
Jean-Marc Vanden-Broeck, \emph{Steep solitary waves in water of finite depth
  with constant vorticity}, J. Fluid Mech. \textbf{274} (1994), 339--348.
  \MR{1297863 (95f:76016)}

\bibitem[VB95]{vandenb:newfamily}
\bysame, \emph{New families of steep solitary waves in water of finite depth
  with constant vorticity}, European J. Mech. B Fluids \textbf{14} (1995),
  no.~6, 761--774. \MR{1364730}

\bibitem[Vol11]{volpert:book}
V.~Volpert, \emph{Elliptic partial differential equations. {V}olume 1:
  {F}redholm theory of elliptic problems in unbounded domains}, Monographs in
  Mathematics, vol. 101, Birkh\"auser/Springer Basel AG, Basel, 2011.
  \MR{2778694}

\bibitem[Wah09]{wahlen:crit}
Erik Wahl{\'e}n, \emph{Steady water waves with a critical layer}, J.
  Differential Equations \textbf{246} (2009), no.~6, 2468--2483. \MR{2498849
  (2010i:76026)}

\bibitem[Wal09]{walsh:stratified}
Samuel Walsh, \emph{Stratified steady periodic water waves}, SIAM J. Math.
  Anal. \textbf{41} (2009), no.~3, 1054--1105. \MR{2529956 (2011a:35430)}

\bibitem[Whe13]{paper}
Miles~H. Wheeler, \emph{Large-amplitude solitary water waves with vorticity},
  SIAM J. Math. Anal. \textbf{45} (2013), no.~5, 2937--2994. \MR{3106477}

\bibitem[Whe15a]{froude}
\bysame, \emph{The {F}roude number for solitary water waves with vorticity},
  Journal of Fluid Mechanics \textbf{768} (2015), 91--112.

\bibitem[Whe15b]{pressure}
\bysame, \emph{Solitary water waves of large amplitude generated by surface
  pressure}, Arch. Ration. Mech. Anal. \textbf{218} (2015), no.~2, 1131--1187.
  \MR{3375547}

\bibitem[WRL95]{wloka:book}
J.~T. Wloka, B.~Rowley, and B.~Lawruk, \emph{Boundary value problems for
  elliptic systems}, Cambridge University Press, Cambridge, 1995. \MR{1343490}

\end{thebibliography}

\end{document}